\documentclass[11pt,twoside,a4paper,dvips]{amsart}

\usepackage{amssymb}
\usepackage{pst-node}
\usepackage{pstricks-add}
\psset{hookwidth=-1.5mm}
\psset{nodesep=2pt,labelsep=1pt,arrowscale=1.2}
\usepackage{enumerate}
\usepackage{ragged2e}
\usepackage[colorlinks=false, backref=true]{hyperref}
\usepackage{extarrows}

\newcommand{\prtline}[4][]{\ncline[#1]{#2}{#3}{#4}\mput{\pscircle[fillstyle=solid,fillcolor=white]{2pt}}}
\newcommand{\prtto}{\Rnode{a}{}\hspace{1cm}\Rnode{b}{}\psset{nodesep=3pt}\everypsbox{\scriptstyle}\prtline{->}{a}{b}}
\newcommand{\injto}{\hookrightarrow}
\newcommand{\epito}{\twoheadrightarrow}
\newcommand{\RA}{{\operatorname{RA}}}
\newcommand{\WKA}{{\operatorname{WKA}}}
\newcommand{\bA}{\mathbf{A}}
\newcommand{\bB}{\mathbf{B}}
\newcommand{\bC}{\mathbf{C}}
\newcommand{\bD}{\mathbf{D}}
\newcommand{\bM}{\mathbf{M}}
\newcommand{\bU}{\mathbb{U}}
\newcommand{\ar}{\operatorname{ar}}
\newcommand{\str}[2]{(#1, (\varrho_{\mathbf{#1}})_{\varrho\in #2})}
\newcommand{\ta}{\overline{a}}
\newcommand{\tb}{\overline{b}}
\newcommand{\tc}{\overline{c}}
\newcommand{\td}{\overline{d}}
\newcommand{\tu}{\overline{u}}
\newcommand{\tx}{\overline{x}}

\newcommand{\dom}{\operatorname{dom}}
\newcommand{\ppTp}{\operatorname{Tpp}}
\newcommand{\bN}{\mathbb{N}}
\newcommand{\Inv}{\operatorname{Inv}}
\newcommand{\Pol}{\operatorname{Pol}}
\newcommand{\End}{\operatorname{End}}
\newcommand{\Age}{\operatorname{Age}}
\newcommand{\cC}{\mathcal{C}}
\newcommand{\cP}{\mathcal{P}}
\newcommand{\pres}{\vartriangleright}
\newcommand{\cE}{\mathcal{E}}
\newcommand{\cL}{\mathcal{L}}
\newcommand{\restr}{\mathord{\upharpoonright}}

\theoremstyle{plain}
\newtheorem{theorem}[subsection]{Theorem}
\newtheorem{proposition}[subsection]{Proposition}

\newtheorem{lemma}[subsection]{Lemma}

\newtheorem{corollary}[subsection]{Corollary}
\theoremstyle{definition}
\newtheorem{definition}[subsection]{Definition}

\theoremstyle{remark}
\newtheorem*{remark}{Remark}

\author[Ch.\,Pech]{Christian Pech}
\address{Institute of Algebra\\Technische Universit\"at Dresden}
\email{Christian.Pech@tu-dresden.de}
\urladdr{http://www.math.tu-dresden.de/~pech} 
\author[M.\,Pech]{Maja Pech}
\address{Department of Mathematics\\University of Novi Sad}
\curraddr{Institute of Algebra\\Technische Universit\"at Dresden}
\email{maja@dmi.uns.ac.rs, Maja.Pech@tu-dresden.de}
\urladdr{http://people.dmi.uns.ac.rs/~maja/}
\title[Polymorphism-homogeneous structures]{On polymorphism-homogeneous relational structures and their clones}
\begin{document}
\begin{abstract}
	A relational structure is ho\-mo\-mor\-phism-ho\-mo\-ge\-ne\-ous if every homomorphism between finite substructures extends to an endomorphism of the structure. This notion was introduced recently by Cameron and Ne\v{s}et\v{r}il. In this paper we consider a strengthening of ho\-mo\-mor\-phism-ho\-mo\-ge\-ne\-i\-ty --- we call a relational structure po\-ly\-mor\-phism-ho\-mo\-ge\-ne\-ous if every partial polymorphism with a finite domain extends to a global polymorphism of the structure. It turns out that this notion (under various names and in completely different contexts) has been existing in algebraic literature for at least 30 years. Motivated by this observation, we dedicate this paper to the topic of po\-ly\-mor\-phism-ho\-mo\-ge\-ne\-ous structures. We study po\-ly\-mor\-phism-ho\-mo\-ge\-ne\-i\-ty from a model-theoretic, an algebraic, and a combinatorial point of view. E.g., we study structures that have quantifier elimination for positive primitive formulae, and show that this notion is equivalent to po\-ly\-mor\-phism-ho\-mo\-ge\-ne\-i\-ty for weakly oligomorphic structures. We demonstrate how the Baker-Pixley theorem can be used to show that po\-ly\-mor\-phism-ho\-mo\-ge\-ne\-i\-ty is a decidable property for finite relational structures. Eventually, we completely characterize the countable po\-ly\-mor\-phism-ho\-mo\-ge\-ne\-ous graphs, the po\-ly\-mor\-phism-ho\-mo\-ge\-ne\-ous posets of arbitrary size, and the countable po\-ly\-mor\-phism-ho\-mo\-ge\-ne\-ous strict posets.  
\end{abstract}

\maketitle

\section{Introduction}\thispagestyle{empty}
A relational structure is called po\-ly\-mor\-phism-ho\-mo\-ge\-ne\-ous if every partial polymorphism with finite domain extends to a global polymorphism of the structure. The phenomenon of po\-ly\-mor\-phism-ho\-mo\-ge\-ne\-i\-ty appears in different contexts under varying names in the algebraic literature. The earliest occurrence of this idea seems to be the Baker-Pixley Theorem in universal algebra \cite{BakPix75} that simultaneously generalizes the Chinese remainder theorem and Langrange's interpolation theorem. In \cite{Rom87},  motivated by questions from multivalued logics and clone theory, Romov studies relational structures over finite sets for which every partial polymorphism can be extended to a global one.  In \cite{Rom08,Rom11} he extended this approach to countably infinite structures. Another source of the notion of po\-ly\-mor\-phism-ho\-mo\-ge\-ne\-i\-ty is \cite{Kaa83}, where Kaarli characterizes all po\-ly\-mor\-phism-ho\-mo\-ge\-ne\-ous meet-complete lattices of equivalence relations and, using this characterization, identifies classes of locally affine complete algebras. Related to po\-ly\-mor\-phism-ho\-mo\-ge\-ne\-i\-ty is also the interpolation condition (IC) that plays an important role in the theory of natural dualities \cite{ClaDav98}. A structure has the (IC) if every partial polymorphism (not necessary with finite domain) extends to a global one. Thus, structures that have the (IC) are in particular  po\-ly\-mor\-phism-ho\-mo\-ge\-ne\-ous. 

The appearance of one and the same idea in such diverse contexts  motivated us to have a closer look onto the theory of po\-ly\-mor\-phism-ho\-mo\-ge\-ne\-ous structures and its connections with other, related, model-theoretic notions like quantifier elimination. Our results extend and generalize previous results by Romov \cite{Rom08}. A diagram that sums up our findings can be found at the end of Section~\ref{sec4}. In Section~\ref{lcRC} we apply the results from Section~\ref{sec4}, to characterize all those countable relational structures that have  the property that the primitively positively definable relations coincide with the relations that are invariant under all polymorphisms of the given structure. As a consequence, we can give a short, model-theoretic proof of Romov's characterization of locally closed relational clones on countable sets \cite[Thm.3.5]{Rom08}.
In Section~\ref{sec5} some connections between the notion of po\-ly\-mor\-phism-ho\-mo\-ge\-ne\-i\-ty and the Baker-Pixley Theorem are presented. In particular, we will  derive an algebraic characterization of the structures that fulfill the interpolation condition among all po\-ly\-mor\-phism-ho\-mo\-ge\-ne\-ous structures. Moreover, we show that for finite structures po\-ly\-mor\-phism-ho\-mo\-ge\-ne\-i\-ty is decidable.   

Recently, in their seminal paper \cite{CamNes06}, Cameron and Ne\v{s}et\v{r}il studied various generalizations of the classical notion of (ultra-)homogeneity (recall that a relational structure is called homogeneous if every isomorphism between finite substructures extends to an automorphism of the given structure). One of the discussed generalizations is ho\-mo\-mor\-phism-ho\-mo\-ge\-ne\-i\-ty. A relational structure is ho\-mo\-mor\-phism-ho\-mo\-ge\-ne\-ous if and only if every homomorphism between finite substructures extends to an endomorphism of the structure.  
It soon turned out that this notion is very relevant in the theory of transformation monoids on countable sets \cite{Dol12,Dol12b,MPPhD,AUpaper,Pon05}. Moreover, there exists already a rich  classification theory for ho\-mo\-mor\-phism-ho\-mo\-ge\-ne\-ous structures \cite{CamLoc10,DolMas11,IliMasRaj08,JunMas12,Mas07,Mas12,MasNenSko11}.
Note that by definition, every po\-ly\-mor\-phism-ho\-mo\-ge\-ne\-ous structure is also ho\-mo\-mor\-phism-ho\-mo\-ge\-ne\-ous. Thus, in a sense, the po\-ly\-mor\-phism-ho\-mo\-ge\-ne\-ous structures are especially beautiful ho\-mo\-mor\-phism-ho\-mo\-ge\-ne\-ous structures. In Section~\ref{sec6}, we will completely characterize the countable po\-ly\-mor\-phism-ho\-mo\-ge\-ne\-ous graphs, the countable po\-ly\-mor\-phism-ho\-mo\-ge\-ne\-ous strict posets, and all po\-ly\-mor\-phism-ho\-mo\-ge\-ne\-ous non-strict posets. At the end of the paper we shortly review Kaarlie's elegant characterization of po\-ly\-mor\-phism-ho\-mo\-ge\-ne\-ous meet-complete lattices of equivalence relations. 

We thank the anonymous referee for the thoughtful remarks that helped tremendously to improve the readability of the paper, and for unearthing a mistake in an earlier version of this text.

\section{Preliminaries}
\subsection*{Relational Structures}
A relational signature $L$ is a family $(\varrho_i)_{i\in I}$ of relational symbols, together with a function $\ar:L\to\bN\setminus\{0\}$ that assigns to every symbol $\varrho$ its arity $\ar(\varrho)$. If $\varrho$ is a relational symbol that belongs to $L$, then we write $\varrho\in L$. A relational structure $\bA$ over the signature $L$ is a pair $(A,(\varrho_\bA)_{\varrho\in L})$ such that $A$ is a set and $\varrho_\bA$ is a relation of arity $\ar(\varrho)$ on $A$. The set $A$ is also called the \emph{carrier} of $\bA$. Relational structures over $L$ will simply be called $L$-structures, or structures (if $L$ is clear from the context). Homomorphisms, endomorphisms, isomorphisms,  automorphisms, epimorphisms, and monomorphisms are defined as usual, for $L$-structures. Embeddings are strong monomorphisms (in the model theoretic sense, cf. \cite{Hod97}). Also, we use the term \emph{substructure} in the model-theoretic sense. In the graph-theoretic terminology, model-theoretic substructures are usually called induced substructures.  Recall that the age of a relational structure $\bA$ is the class of all finite structures that embed into $\bA$. In other words, a structure is in the age of $\bA$ if and only if it is isomorphic to some finite substructure of $\bA$. The age of $\bA$ will be denoted by $\Age(\bA)$.

\subsection*{Polymorphism-homogeneity}
Let $\bA$ and $\bB$ be relational structures over the same signature $L$.  Let $\bD\le\bA$. Then a homomorphism $f:\bD\to\bB$ is called \emph{partial homomorphism} from $\bA$ to $\bB$ with domain $\bD$ (written as $f:\bA\prtto\bB$).
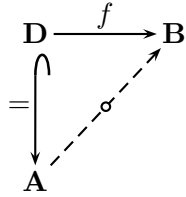
\par
\[
  \begin{psmatrix}
    \bD & \bB\\
    \bA
   \ncline{->}{1,1}{1,2}^{f}
   \ncline{H->}{1,1}{2,1}\Bput{=}
   \prtline[linestyle=dashed]{->}{2,1}{1,2}
  \end{psmatrix}
\]
The structure $\bD$ will usually be denoted by $\dom f$. In the special case where $L$ is the empty signature, partial homomorphisms are just usual partial functions. Partial homomorphisms of a structure to itself are called \emph{partial endomorphisms}. Finally, partial homomorphisms with a finite domain will be called \emph{local homomorphisms}. 

Let $I$ be a set and, for $i\in I$,  let $\bA_i=(A_i,(\varrho_{\bA_i})_{\varrho\in L})$ be relational structures. The the \emph{product} of the family $(\bA_i)_{i\in I}$ is defined by
\[\bA=\prod_{i\in I} \bA_i := \left(\prod_{i\in I} A_i, (\varrho_\bA)_{\varrho\in L}\right),\]
where
\[\prod_{i\in I} A_i := \{(a_i)_{i\in I}\mid \forall i\in I\, : \,a_i\in A_i\}\]
and such that for all $\varrho\in L$ we have
\[\varrho_\bA := \{((a_{1,i})_{i\in I}, \ldots, (a_{\ar(\varrho),i})_{i\in I}) \mid\forall i\in I\,:\, (a_{1,i},\ldots, a_{\ar(\varrho),i})\in\varrho_{\bA_i}\}.\]
If $\emptyset\subsetneq J\subseteq I$, then we denote the \emph{projection homomorphism} with respect to $J$ by
\[e_J: \prod_{i\in I} \bA_i \to \prod_{i\in J} \bA_i\qquad (a_i)_{i\in
  I}\mapsto (a_i)_{i\in J}.\] 
In the special case, when $J=\{j\}$, we write $e_j$ instead of $e_{\{j\}}$.

When all $\bA_i$ are equal to one and the same structure $\bA$, then we abbreviate the product $\prod_{i\in I} \bA_i$ by $\bA^I$. This special kind of direct product is called a \emph{direct power} of $\bA$. Direct powers will usually occur for the special case $I=k$, where $k$ is a finite cardinal number.

Let $L$ be a relational signature and let  $\bA=\str{A}{L}$ be an $L$-structure. Then the \emph{$k$-ary polymorphisms} of $\bA$ are defined to be the homomorphisms from $\bA^k$ to $\bA$. Partial and local $k$-ary polymorphisms of $\bA$ are defined accordingly, as partial or local homomorphisms from $\bA^k$ to $\bA$, respectively. The set of all polymorphisms of $\bA$ will be denoted by $\Pol(\bA)$, while the set of all $k$-ary polymorphisms will be denoted by $\Pol^{(k)}(\bA)$. The unary polymorphisms of $\bA$ are called \emph{endomorphisms} of $\bA$. Traditionally, the set of endomorphisms of $\bA$ is denoted by $\End(\bA)$. 

It is not hard to see that $k$-ary partial polymorphisms of $\bA$ are characterized by the following property:
  A partial function $f:A^k\prtto A$ is a partial polymorphism of $\bA$ if and only if for all $\varrho\in L$ and for all $\ta_1,\ldots,\ta_{\ar(\varrho)}\in\dom f$ with $\ta_i=(a_{i,1},\ldots,a_{i,k})$ holds that
\[ 
	\begin{bmatrix}
		a_{1,1} \\
		\vdots \\
		a_{\ar(\varrho),1} \\
	\end{bmatrix}\in\varrho_\bA,\ldots,
	\begin{bmatrix}
		a_{1,k} \\
		\vdots \\
		a_{\ar(\varrho),k} \\
	\end{bmatrix}\in\varrho_\bA \Longrightarrow
	\begin{bmatrix}
		f(a_{1,1},\ldots,a_{1,k}) \\
		\vdots \\
		f(a_{\ar(\varrho),1},\ldots,a_{\ar(\varrho),k}) \\
	\end{bmatrix}\in\varrho_\bA\]

We say that a relational structure $\bA$ is \emph{$k$-po\-ly\-mor\-phism-ho\-mo\-ge\-ne\-ous} if every $k$-ary local polymorphism of $\bA$ can be extended to a polymorphism of $\bA$.  If $\bA$ is $k$-po\-ly\-mor\-phism-ho\-mo\-ge\-ne\-ous for every $k\in\bN\setminus\{0\}$, then we say that $\bA$ is \emph{po\-ly\-mor\-phism-ho\-mo\-ge\-ne\-ous}.  If a structure $\bA$ is $1$-po\-ly\-mor\-phism-ho\-mo\-ge\-ne\-ous then we call it \emph{ho\-mo\-mor\-phism-ho\-mo\-ge\-ne\-ous}. Thus, every po\-ly\-mor\-phism-ho\-mo\-ge\-ne\-ous structure is ho\-mo\-mor\-phism-ho\-mo\-ge\-ne\-ous. However, po\-ly\-mor\-phism-ho\-mo\-ge\-ne\-i\-ty is a much stronger  property than ho\-mo\-mor\-phism-ho\-mo\-ge\-ne\-i\-ty, as can be seen from the following simple observation:
\begin{proposition}\label{PHHH}
	A structure $\bA$ is $k$-po\-ly\-mor\-phism-ho\-mo\-ge\-ne\-ous if and only if $\bA^k$ is ho\-mo\-mor\-phism-ho\-mo\-ge\-ne\-ous.\qed
\end{proposition}
\begin{proof}
	``$\Rightarrow$'' Let $h:\bA^k\prtto\bA^k$ be a local homomorphism of $\bA^k$, with domain $\bD$. For $0\le i<k$, define $h_i:=e_i\circ h$ --- i.e., if $h(a_0,\dots,a_{k-1}) = (b_0,\dots,b_{k-1})$, then $h_i(a_0,\dots,a_{k-1})=b_i$. Then $h_i$ is a  $k$-ary local polymorphism of $\bA$, for every $i\in\{0,\dots,k-1\}$. Since  $\bA$ is $k$-polymorphism-homogeneous, for every $i\in\{0,\dots,k-1\}$ there exists a $\hat{h}_i\in\Pol^{(k)}(\bA)$ that extends $h_i$. It is now easy to see that $\hat{h}:A^k\to A^k$ defined through
	\[ \hat{h}(\ta):=\big(\hat{h}_0(\ta),\dots,\hat{h}_{k-1}(\ta)\big)\text{ where } \ta=(a_0,\dots,a_{k-1})\in A^k\]
	is an endomorphism of $\bA^k$ that extends $h$.
	
	``$\Leftarrow$'' Let $h:\bA^k\prtto\bA$ be a local $k$-ary polymorphism of $\bA$. Define $h':\bA^k\prtto\bA^k$ by $\ta\mapsto(h(\ta),\dots,h(\ta))$, for every $\ta$ from the domain of $h$. Since $\bA^k$ is homomorphism homogeneous $h'$ extends to an endomorphism $\tilde{h}$ of $\bA^k$. Now define $\hat{h}:\bA^k\to \bA$ according to $\hat{h}:=e_1 \circ \tilde{h}$. Then $\hat{h}$ extends $h$.  
\end{proof}

An immediate consequence is that a relational structure $\bA$ is po\-ly\-mor\-phism-ho\-mo\-ge\-ne\-ous if and only if all finite powers of $\bA$ are ho\-mo\-mor\-phism-ho\-mo\-ge\-ne\-ous.

The following proposition will show that the concept of $k$-po\-ly\-mor\-phism-ho\-mo\-ge\-ne\-i\-ty defines a decreasing hierarchy on the ho\-mo\-mor\-phism-ho\-mo\-ge\-ne\-ous structures with ho\-mo\-mor\-phism-ho\-mo\-ge\-ne\-ous structures on the top.
\begin{proposition}
	Let $\bA$ be a relational structure, and let   $k\in\bN\setminus\{0\}$. If $\bA$ is $(k+1)$-po\-ly\-mor\-phism-ho\-mo\-ge\-ne\-ous, then it is also $k$-po\-ly\-mor\-phism-ho\-mo\-ge\-ne\-ous.
\end{proposition}
\begin{proof}
	Let $f$ be a $k$-ary local polymorphism of $\bA$ with domain $\bD$. Moreover, let $\bD_k$ be the (finite) substructure of $\bA$ that is induced by 
	\[
	e_{k-1}(D) = \{d_{k-1}\mid \exists d_0,\dots,d_{k-2} : (d_0,\dots,d_{k-1})\in D\}.
	\] 
	First we construct a $(k+1)$-ary local polymorphism $\hat{f}: \bA^{k+1}\prtto\bA$ with domain $\bD\times\bD_k$ by setting
	\[ \hat{f}(a_0,\ldots, a_{k-1},x):= f(a_0,\ldots,a_{k-1})\qquad\text{for all $x\in D_k$.}\] 
	More precisely, $\hat{f}$ is the unique homomorphism that makes the following diagram commutative:
	\begin{equation}
	\begin{psmatrix}
		\bA^{k}   & \bD          & \bA\\
		\bA^{k+1} & \bD\times\bD_k
		\ncline[hookwidth=1.5mm]{H->}{1,2}{1,1}^{=}
		\ncline[hookwidth=1.5mm]{H->}{2,2}{2,1}^{=}
		\ncline{->>}{2,1}{1,1}<{e_{\{0,\ldots,k-1\}}}
		\ncline{->>}{2,2}{1,2}<{e}
		\ncline{->}{1,2}{1,3}^{f}
		\ncline[linestyle=dashed]{->}{2,2}{1,3}_{\hat{f}}
	\end{psmatrix}
	\end{equation}
	where $e:\bD\times\bD_k\epito\bD$ is the projection homomorphism onto $\bD$.
	So in particular $\hat{f}=f\circ e$, and $\dom\hat{f}$ is finite. Since $\bA$ is $(k+1)$-po\-ly\-mor\-phism-ho\-mo\-ge\-ne\-ous, $\hat{f}$ can be extended to a polymorphism $\hat{g}$ of $\bA$. Now we can define a polymorphism $g: \bA^k\to \bA$: For this consider the embedding
	\[\iota : \bA^k \injto \bA^{k+1}\quad (a_0,\ldots,a_{k-1})\mapsto (a_0,\ldots, a_{k-1},a_{k-1}).\]
	and its restriction $\iota\restr_\bD: \bD\to\bD\times\bD_k$. Note that $\iota$ and $\iota\restr_\bD$ are right-inverses of $e_{\{0,\ldots,k-1\}}$ and $e$, respectively. Hence the following diagram commutes:
	\[
	\begin{psmatrix}
		\bA^{k}   & \bD          & \bA\\
		\bA^{k+1} & \bD\times\bD_k
		\ncline[hookwidth=1.5mm]{H->}{1,2}{1,1}^{=}
		\ncline[hookwidth=1.5mm]{H->}{2,2}{2,1}^{=}
		\ncline{<-H}{2,1}{1,1}<{\iota}
		\ncline{<-H}{2,2}{1,2}<{\iota\restr_\bD}
		\ncline{->}{1,2}{1,3}^{f}
		\ncline[linestyle=dashed]{->}{2,2}{1,3}_{\hat{f}}
		\ncarc[arcangle=-50,ncurv=1]{->}{2,1}{1,3}_{\hat{g}}
	\end{psmatrix}
	\]
	\vspace*{1cm}

	Now, with $g=\hat{g}\circ\iota$, and using that the above diagram commutes, we obtain that $g$ is indeed an extension of $f$ to a polymorphism of $\bA$.
\end{proof}

Let $\bA$ be a relational structure. Then $\bA$ is called \emph{weakly po\-ly\-mor\-phism-ho\-mo\-ge\-ne\-ous} if for all $n\in\bN\setminus\{0\}$, for every finite structure $\bC\le\bA^n$, for all $\bB\le\bC$ and for all homomorphisms $h:\bB\to\bA$ there exists a homomorphism $\hat{h}:\bC\to\bA$ that extends $h$ to $\bC$. In other words, the following diagram commutes:
\[
\begin{psmatrix}
	[name=B]\bB & [name=A]\bA\\
	[name=C]\bC
	\ncline{H->}{B}{C}<{=}
	\ncline{->}{B}{A}^{h}
	\ncline{->}{C}{A}_{\hat{h}}
\end{psmatrix}
\]
\begin{lemma}
	Po\-ly\-mor\-phism-ho\-mo\-ge\-ne\-i\-ty implies weak po\-ly\-mor\-phism-ho\-mo\-ge\-ne\-i\-ty. Moreover, for countable relational structures the concepts of po\-ly\-mor\-phism-ho\-mo\-ge\-ne\-i\-ty and weak po\-ly\-mor\-phism-ho\-mo\-ge\-ne\-i\-ty coincide.
\end{lemma}
\begin{proof}
	Let $\bA$ be a polymorphism-homogeneous structure, let $n\in\bN\setminus\{0\}$, let $\bC$ be a finite substructure of $\bA^n$, and let $\bB\le\bC$. Finally, let $h:\bB\to\bA$. Then $h$ defines a local  $n$-ary polymorphism of $\bA$ with domain $\bB$. Since $\bA$ is polymorphism-homogeneous, $h$ can be extended to an $n$-ary polymorphism of $\bA$. Now we take $\hat{h}:=g\restr_\bC$, and obtain at once that $\hat{h}$ extends $h$ to $\bC$. 
	
	Let now $\bA$ be a countable weakly polymorphism-homogeneous relational structure. 
	Let $f:\bA^k\prtto\bA$ be a local polymorphism of $\bA$. Since $\bA^k$ is countable, it can be written as a union of a chain $(\bD_i)_{i\in\bN}$ of finite substructures. Without loss of generality, we may assume that $\bD_0$ is the domain of $f$. Now, inductively, we construct a series $(h_i)_{i\in\bN}$ such that $h_i:\bD_i\to\bA$ and such that whenever $j>i$, then $h_j\restr_{\bD_i}=h_i$. We set $h_0:= f$. Now, if $h_i$ has already been defined, then, by the weak polymorphism homogeneity of $\bA$, $h_i$ can be extended to a homomorphism from $\bD_{i+1}$ to $\bA$. Take any such extension and call it $h_{i+1}$. Thus the desired chain of homomorphisms is defined. Finally, we define $h$ to be the union over all $h_i$. Now it is easy to see that $h$ is indeed a $k$-ary polymorphism of $\bA$ that extends $f$. 
\end{proof}

\section{The Galois connections between relations and positive primitive types}\label{sec4}
\subsection*{Galois Connections}
Let $\bA=(A,\le_\bA)$ and $\bB=(B,\le_\bB)$ be partially ordered sets, and let $\alpha: A\to B$ and $\beta_B\to A$ be functions. Then the pair $(\alpha,\beta)$ is called \emph{Galois connection} between $\bA$ and $\bB$ if for all $a\in A$ and for all $b\in B$ it holds that 
\[a\le_\bA\beta(b) \iff b\le_\bB\alpha(a).\]
In this paper, we will restrict our attention to the case where $\bA=(\cP(G),\subseteq)$ and $\bB=(\cP(M),\subseteq)$. Following G.~Birkhoff \cite{Bir40}, Galois connections in this special setting are called \emph{polarities}. It was shown by Ore \cite{Ore44}, that for every polarity $(\alpha,\beta)$ between $\bA$ and $\bB$ there exists a relation $I\subseteq G\times M$ such that for every $S\subseteq G$ and for every $T\subseteq M$ we have 
\begin{align*}
	\alpha(S)&= \{m\in M\mid \forall g\in S: (g,m)\in I\},\\
	\beta(T) &= \{g\in G\mid \forall m\in T: (g,m)\in I\}.
\end{align*} 
Moreover, every relation $I\subseteq G\times M$ induces a polarity between $\bA$ and $\bB$ in this way. By this every polarity between $\bA$ and $\bB$ is completely determined by the triple $(G,M,I)$. According to \cite{GanWil99}, such triples are also called \emph{(formal) contexts}. 

If $(\alpha,\beta)$ is a polarity between $\bA$ and $\bB$, then $\alpha\circ\beta$ and $\beta\circ\alpha$ are closure operators on $M$ and $G$. The corresponding closed sets are called \emph{Galois-closed sets} of $(\alpha,\beta)$. The Galois-closed subsets of $G$ are also called \emph{extents} and the Galois-closed subsets of $M$ are called \emph{intents}. Both, the set of extents and the set of intents of $(\alpha,\beta)$ form \emph{closure-systems} --- i.e., they are closed with respect to arbitrary intersections, $G$ is an extent,  and $M$ is an intent.  For an indepth treatment of the theory of Galois connections, we refer the reader to \cite{DenErnWis04}.

\subsection*{Positive primitive types} 
Let $\bA=\str{A}{L}$ be a relational structure over the relational signature $L$. Following \cite{Hod97}, with $L_{\omega\omega}$ we denote the full first order language with predicate symbols from $L$. 

As usual, every formula $\varphi(x_1,\dots,x_m)$ from $L_{\omega\omega}$ defines an $m$-ary relation $\varphi^\bA$ on $A$ where $\ta\in\varphi^\bA$ if $\varphi$ holds in $\bA$ when interpreting $x_i$ through $a_i$, for $i=1,\dots,m$. 
In that case, we say that $\ta$ fulfills $\varphi$ in $\bA$ and write $\bA\models\varphi(\ta)$. 

A set of positive primitive formulae in $L_{\omega\omega}$ with free variables in $\{x_1,\dots,x_m\}$ is called an \emph{$m$-ary positive primitive type}. For a relation $\varrho\subseteq A^m$ we define
$\ppTp_\bA(\varrho)$ to be the set of all positive primitive formulae $\varphi(x_1,\dots,x_m)$ from $L_{\omega\omega}$ such that $\varrho\subseteq \varphi^\bA$. On the other hand, for an $m$-ary positive primitive type $\Psi$, we define  
\[	\Psi^\bA := \bigcap_{\varphi\in\Psi}\varphi^\bA.\]
The operators $\varrho\mapsto\ppTp_\bA(\varrho)$ and $\Psi\mapsto\Psi^\bA$ form a Galois-connection between the $m$-ary relations on $A$ and the $m$-ary positive primitive types over $L_{\omega\omega}$. The Galois-closed types are called \emph{closed positive primitive types} of $\bA$. 

A closed $m$-ary positive primitive type over $\bA$ is called \emph{principal} if it is generated by one of its elements. Types that are of the shape $\ppTp_\bA(\tau)$ for some finite relation $\tau\subseteq A^m$ are called \emph{complete $m$-ary positive primitive types} over $\bA$.

If $\Psi$ is an $m$-ary positive primitive type, then by $\Psi^{(k)}$ we denote the set of all formulae from $\Psi$ that are logically equivalent to a positive primitive formula of of the shape
\[ \exists x_{m+1},\dots, x_{m+k} \psi(x_1,\dots,x_{m+k}),\]
where $\psi(x_1,\dots,x_{m+k})$ is a conjunction of atoms. Instead of $(\ppTp_\bA(\tau))^{(k)}$ we will write also $\ppTp_\bA^{(k)}(\tau)$.

If $\Psi$ is a principal positive primitive type over $\bA$, then the relation $\Psi^\bA$ is called \emph{positive primitively definable} over $\bA$. The set of all positive primitively definable relations (of all arities) over $\bA$ is called the \emph{relational algebra} generated by $\bA$. It is denoted by $[\bA]_\RA$, while its $m$-ary part is denoted by $[\bA]_\RA^{(m)}$. 

In general, the set of Galois-closed relations (of arbitrary arities) over $\bA$ will be denoted by $\overline{[\bA]_\RA}$. 

\subsection*{Algebraic closure systems}
Recall that a closure system $\cC$ on a set $A$ is called \emph{algebraic} if there exists some algebraic structure $\bA$ with carrier $A$ whose closure system of universes of subalgebras is equal to $\cC$.   

A closure system is called \emph{inductive} if it is closed with respect to unions of chains of closed sets. There is another equivalent way to define inductive closure systems. A subset $M$ of a given closure system is called \emph{upwards directed} if any two elements of $M$ are contained in a third. Now we have that a closure system is inductive if and only if it is closed with respect to unions of upwards directed sets of closed sets (cf. \cite{MaySte64}).

We have the following classical result that characterizes algebraic closure systems (cf. \cite{Sch52}):
\begin{theorem}[{Schmidt \cite[Hauptsatz]{Sch52}}]\label{schmidt}
	Let $\cC$ be a closure system on a set $A$, and let $C:\cP(A)\to\cC$ be the associated closure operator. Then the following are equivalent:
	\begin{enumerate}
		\item $\cC$ is an algebraic closure system,
		\item $\cC$ is an inductive closure system,
		\item $\forall S\subseteq A: C(S)=\bigcup\limits_{\substack{T\subseteq S\\T\text{ finite}}} C(T)$.\qed
	\end{enumerate}
\end{theorem}

\subsection*{Weak oligomorphy and related properties}
Let $\bA$ be a relational structure. We say that 
\begin{itemize}
	\item  $\bA$ is \emph{weakly oligomorphic}  if for all $m\in \mathbb{N}\setminus\{0\}$ there are just finitely many $m$-ary positive primitively definable relations over $\bA$,
	\item $\bA$ is \emph{algebraic} if for all $m\in\bN$ the closure-system $\overline{[\bA]_\RA^{(m)}}$ is algebraic,
	\item $\bA$ is \emph{pp-atomic} if all complete positive primitive types over $\bA$ are principal,
	\item $\bA$ has \emph{principal positive primitive types}, if all closed positive primitive types over $\bA$ are principal.
\end{itemize}

\begin{lemma}\label{lem:add_imp}
	All weakly oligomorphic relational structures are algebraic, pp-atomic, and have principal positive primitive types. Moreover, if a relational structure has principal positive primitive types, then it is also pp-atomic. 
\end{lemma}
\begin{proof}
	straightforward.
\end{proof}

\begin{lemma}\label{princint}
	Let $\bA$ be a relational structure. Then the following are equivalent.
	\begin{enumerate}
		\item $\bA$ has principal positive primitive types.
		\item $[\bA]_\RA= \overline{[\bA]_\RA}$.
		\item $[\bA]^{(m)}_\RA$ is closed with respect to arbitrary intersections for every $m\in\bN\setminus\{0\}$.
	\end{enumerate}
\end{lemma}
\begin{proof}
	$(1)\Rightarrow (2)$ Suppose that $\bA$ has principal positive primitive types. Note that $[\bA]^{(m)}_\RA\subseteq\overline{[\bA]^{(m)}_\RA}$ always holds. Let $\sigma\in\overline{[\bA]^{(m)}_\RA}$, and let $\Psi=\ppTp_\bA(\sigma)$. Since $\Psi$ is closed, it follows that there exists a $\psi\in\Psi$ such that $\ppTp_\bA(\psi^\bA)=\Psi$. However, from this it follows that $\psi^\bA=\Psi^\bA=\sigma$. Hence $\sigma\in[\bA]^{(m)}_\RA$.

	$(2)\Rightarrow (3)$ Take any family $\{\sigma_i\}_{i\in I}$ of elements of $[\bA]^{(m)}_\RA$. We know that
	\[\sigma:=\bigcap_{i\in I}\sigma_i\in\overline{[\bA]^{(m)}_\RA}.\]
	So from the premise it follows that $\sigma\in[\bA]^{(m)}_\RA$.

	$(3)\Rightarrow (1)$ Let $\Psi$ be an arbitrary closed positive primitive type over $\bA$, and let $\sigma=\Psi^\bA$. Then
	\[\sigma=\bigcap_{\psi\in\Psi}\psi^\bA.\]
	Since $[\bA]^{(m)}_\RA$ is closed with respect to arbitrary intersections, we get that $\sigma\in[\bA]^{(m)}_\RA$. Hence, there exists a formula $\psi\in\Psi$ such that $\sigma=\psi^\bA$. Thus, $\Psi^\bA=\sigma=\psi^\bA$. This shows that $\Psi$ is principal.
\end{proof}

\subsection*{Polylocality} A relational structure $\bA$ will be called \emph{$k$-polylocal} if for every $m\in\bN$ and for every $\sigma\subseteq A^m$ holds
\begin{multline*}
	\sigma\in\overline{[\bA]^{(m)}_\RA}\iff \forall\tau\subseteq \sigma \text{ finite}, \tb\in A^m \,: \ppTp_\bA^{(k)}(\tau)\subseteq\ppTp_\bA^{(k)}(\tb)\Rightarrow \tb\in\sigma.
\end{multline*}
If $\bA$ is $0$-polylocal, then we call it \emph{polylocal}. 
\begin{lemma}\label{localg}
	If a relational structure $\bA$ is $k$-polylocal, then it is algebraic.
\end{lemma}
\begin{proof}
	Let $m\in\bN\setminus\{0\}$. The closure operator $C$ on $\overline{[\bA]_\RA^{(m)}}$ is given by
	\[ C: \sigma\mapsto \big(\ppTp_\bA(\sigma)\big)^\bA.\]
	Let $\sigma\subseteq A^m$. Suppose that 
	\[ 
	  \hat\sigma:=\bigcup_{\substack{\tau\subseteq\sigma\\\tau\text{ finite}}} C(\tau)
	\] 
	is not in $\overline{[\bA]_\RA^{(m)}}$. Then, since $\bA$ is $k$-polylocal, there exists a finite $\tau\subseteq\hat\sigma$ and some $\tb\in A^m\setminus\hat\sigma$, such that $\ppTp_\bA^{(k)}(\tau)\subseteq\ppTp_\bA(\tb)$. Since $\tau$ is finite, there exists a finite $\tilde\tau\subseteq\sigma$ such that $\tau\subseteq C(\tilde\tau)$. Again using that $\bA$ is $k$-polylocal, we obtain that $\tb\in C(\tilde\tau)\subseteq\hat\sigma$ --- contradiction. It follows that $\hat\sigma$ is equal to $C(\sigma)$, whence $\bA$ is algebraic.
\end{proof}

\begin{lemma}\label{l34}
	A relational structure $\bA$ is $k$-polylocal if and only if for every $\sigma\subseteq A^m$ holds
	\[ \sigma\in\overline{[\bA]^{(m)}_\RA}\iff \sigma=\bigcup_{\substack{\tau\subseteq\sigma\\\tau\text{ finite}}}\big(\ppTp_\bA^{(k)}(\tau)\big)^\bA.\]
\end{lemma}
\begin{proof}
	Let $\sigma\subseteq A^m$. We show:
	\begin{gather}
		\label{eq1}\forall\tau\subseteq \sigma \text{ finite}, \tb\in A^m \,: \ppTp_\bA^{(k)}(\tau)\subseteq\ppTp_\bA^{(k)}(\tb)\Rightarrow \tb\in\sigma \\\notag\Updownarrow\\
		\label{eq2}\sigma=\bigcup_{\substack{\tau\subseteq\sigma\\\tau\text{ finite}}}\big(\ppTp_\bA^{(k)}(\tau)\big)^\bA.
	\end{gather}
	``$\Rightarrow$'' Suppose, \eqref{eq1} holds. By the properties of Galois connections, we have $\tau\subseteq \big(\ppTp_\bA^{(k)}(\tau)\big)^\bA$, for all $\tau\subseteq A^m$. Thus we have 
	\[
	\sigma= \bigcup_{\substack{\tau\subseteq\sigma\\\tau\text{ finite}}}\tau\subseteq \bigcup_{\substack{\tau\subseteq\sigma\\\tau\text{ finite}}}\big(\ppTp_\bA^{(k)}(\tau)\big)^\bA.
	\] Let now $\tb\in \bigcup\limits_{\substack{\tau\subseteq\sigma\\\tau\text{ finite}}}\big(\ppTp_\bA^{(k)}(\tau)\big)^\bA$. Then there exists a finite $\tau\subseteq\sigma$, such that $\tb\in\big(\ppTp_\bA^{(k)}(\tau)\big)^\bA$. Using the  basic properties of Galois connections we obtain 
	\[ \ppTp_\bA^{(k)}(\tb)\supseteq \ppTp_\bA^{(k)}\Big(\big(\ppTp_\bA^{(k)}(\tau)\big)^\bA\Big) = \ppTp_\bA^{(k)}(\tau). \]
	From assumption \eqref{eq1}, we conclude that $\tb\in\sigma$. Thus \eqref{eq2} holds.
	
	``$\Leftarrow$''  Suppose, \eqref{eq2} holds. Let $\tau\subseteq\sigma$ be finite and let $\tb\in A^m$, such that $\ppTp_\bA^{(k)}(\tau)\subseteq\ppTp_\bA^{(k)}(\tb)$. Then, by the basic properties of Galois connections, we have 
	\[\tb\in \big(\ppTp_\bA^{(k)}(\tb)\big)^\bA\subseteq \big(\ppTp_\bA^{(k)}(\tau)\big)^\bA.\]
	By the assumption we have $\big(\ppTp_\bA^{(k)}(\tau)\big)^\bA\subseteq\sigma$. Thus we conclude $\tb\in\sigma$. Hence \eqref{eq1} holds.
\end{proof}

\begin{lemma}\label{mainlemma}
	Let $\bA$ be a relational structure. Then the following are equivalent:
	\begin{enumerate}
		\item $\bA$ is $k$-polylocal,
		\item $\bA$ is algebraic and for all $m\in\bN\setminus\{0\}$, and for all finite $\tau\subseteq A^m$ we have
		\[ \big(\ppTp_\bA(\tau)\big)^\bA= \big(\ppTp_\bA^{(k)}(\tau)\big)^\bA\]
	\end{enumerate}
\end{lemma}
\begin{proof}
	$(1)\Rightarrow (2):$ From Lemma~\ref{localg} it follows that $\bA$ is algebraic. Let $m\in\bN\setminus\{0\}$, $\tau\subseteq A^m$ finite. 
	By the definition of $k$-polylocality, we have
	\[ (\ppTp_\bA^{(k)}(\tau))^\bA\subseteq (\ppTp_\bA(\tau))^\bA. \]
	On the other hand, $\ppTp_\bA^{(k)}(\tau)\subseteq\ppTp_\bA(\tau)$ implies that $(\ppTp_\bA^{(k)}(\tau))^\bA\supseteq (\ppTp_\bA(\tau))^\bA$. Thus, (2) is proved.
	
	$(2)\Rightarrow (1)$ Since $\bA$ is algebraic, we have.
	\[ \sigma\in\overline{[\bA]^{(m)}_\RA}\iff\sigma=\bigcup_{\substack{\tau\subseteq\sigma\\\tau\text{ finite}}}(\ppTp_\bA(\tau))^\bA\]
	By assumption we may replace in each term of the union $\ppTp_\bA(\tau)$ by $\ppTp_\bA^{(k)}(\tau)$. This together with Lemma~\ref{l34} proves the claim. 
\end{proof}

\begin{proposition}\label{wPHPL}
	Let $\bA$ be an algebraic, weakly po\-ly\-mor\-phism-ho\-mo\-ge\-ne\-ous relational structure. Then $\bA$ is polylocal.
\end{proposition}
\begin{proof}
	We are going to use Lemma~\ref{mainlemma} in order to show that $\bA$ is polylocal. Let $\tau\subseteq A^m$ be finite. In particular, $\tau=\{\ta_1,\dots,\ta_l\}$, for some $l\in\bN$, with $\ta_i=(a_{i,1},\dots,a_{i,m})$.
	
	Let $\tb=(b_1,\dots,b_m)\in \big(\ppTp_\bA^{(0)}(\tau)\big)^\bA$, i.e., $\ppTp_\bA^{(0)}(\tau)\subseteq\ppTp_\bA^{(0)}(\tb)$. 
	Let us show that then also $\ppTp_\bA(\tau)\subseteq\ppTp_\bA(\tb)$. For some $k\in\bN$ let $\psi\in\ppTp_\bA^{(k)}(\tau)$. 
	It is well known that there exists a finite structure $\bD$ and a tuple $\td\in D^m$ such that for all $\ta\in A^m$ we have that $\ta\in\psi^\bA$ if and only if there is a homomorphism $f:(\bD,\td)\to(\bA,\ta)$ --- i.e., $f:\bD\to\bA$ and $f(d_i)=a_i$, for $i\in\{1,\dots,m\}$.  Since $\tau\subseteq\psi^\bA$, it follows that there exists a homomorphism $g:(\bD,\td)\to(\bA^l,\tc)$, where $\tc=(\tc_1,\dots,\tc_m)$ and where $\tc_j=(a_{1,j},\dots, a_{l,j})$ for $1\le j\le m$. 
	
	Let now $\widetilde\bM=g(\bD)$, $\tilde{g}:\bD\epito\widetilde\bM$, such that the following diagram commutes:
	\[
		\begin{psmatrix}
			[name=D] \bD & [name=Al]\bA^l\\
			[name=tM]\widetilde\bM
			\ncline{->}{D}{Al}^{g}
			\ncline{->>}{D}{tM}<{\tilde{g}}
			\ncline{H->}{tM}{Al}_=
		\end{psmatrix}
	\]
	Let $M=\{\tc_1,\dots,\tc_m\}$, and $B=\{b_1,\dots,b_m\}$. Let $\bM$ be the substructure of $\bA^m$ induced by $M$, and let $\bB$ be the substructure of $\bA$ induced by $B$. Since $\ppTp_\bA^{(0)}(\tau)\subseteq\ppTp_\bA^{(0)}(\tb)$, the function $\tilde{f}:\bM\to\bB$ defined by $\tilde{f}: \tc_i\mapsto b_i$ for $1\le i\le m$, is an epimorphism. Let $f$ be the homomorphism that makes the following diagram commutative:
	\[
		\begin{psmatrix}
			[name=M] \bM & [name=A]\bA\\
			[name=B]\bB
			\ncline{->}{M}{A}^{f}
			\ncline{->>}{M}{B}<{\tilde{f}}
			\ncline{H->}{B}{A}_=
		\end{psmatrix}
	\]
	Since $\bA$ is weakly po\-ly\-mor\-phism-ho\-mo\-ge\-ne\-ous, it follows that $f$ has an extension to $\widetilde\bM$ --- call it $h$. Let $\widetilde\bB=h(\widetilde\bM)$, and let $\tilde{h}$ be the homomorphism that makes the following diagram commutative:
	\[
		\begin{psmatrix}
			[name=tM] \widetilde\bM & [name=A]\bA\\
			[name=B]\widetilde\bB
			\ncline{->}{tM}{A}^{h}
			\ncline{->>}{tM}{B}<{\tilde{h}}
			\ncline{H->}{B}{A}_=
		\end{psmatrix}
	\]
	Then, by construction, the following diagram commutes:
	\[
		\begin{psmatrix}
			[name=tM] (\widetilde\bM,\tc) & [name=A](\bA,\tb)\\
			[name=B](\widetilde\bB,\tb)
			\ncline{->}{tM}{A}^{h}
			\ncline{->>}{tM}{B}<{\tilde{h}}
			\ncline{H->}{B}{A}_=
		\end{psmatrix}
	\]
	Summing up our findings, we obtain that also the following diagram commutes:
	\[
	\begin{psmatrix}
		[name=D](\bD,\td) & [name=Al](\bA^l,\tc)\\
		[name=tM](\widetilde\bM,\tc) \\
		[name=B] (\widetilde\bB,\tb) & [name=A](\bA,\tb)
		\ncline{->}{D}{Al}^{g}
		\ncline{->>}{D}{tM}<{\tilde{g}}
		\ncline{->>}{tM}{B}<{\tilde{h}}
		\ncline{->}{tM}{A}^{h}
		\ncline{H->}{B}{A}^{=}
		\ncline{H->}{tM}{Al}^{=}
	\end{psmatrix}
	\]
	In particular, $h\circ\tilde{g}:(\bD,\td)\to(\bA,\tb)$. Hence, $\tb\in\psi^\bA$. Consequently, $\psi\in\ppTp_\bA(\tb)$, whence also $\ppTp_\bA(\tau)\subseteq\ppTp_\bA(\tb)$, i.e.\ $\tb\in\big(\ppTp_\bA(\tau)\big)^\bA$. It follows that $\big(\ppTp_\bA^{(0)}(\tau)\big)^\bA\subseteq\big(\ppTp_\bA(\tau)\big)^\bA$. Moreover, from $\ppTp_\bA^{(0)}(\tau)\subseteq\ppTp_\bA(\tau)$ it follows that $\big(\ppTp_\bA(\tau)\big)^\bA\subseteq\big(\ppTp_\bA^{(0)}(\tau)\big)^\bA$. Now from Lemma~\ref{mainlemma} it follows that $\bA$ is polylocal. 
\end{proof}

\begin{proposition}\label{watwPH}
	Let $\bA$ be a pp-atomic relational structure that is polylocal or has the property that every complete positive primitive type $\Psi$ over $\bA$ is generated by $\Psi^{(0)}$. Then $\bA$ is weakly po\-ly\-mor\-phism-ho\-mo\-ge\-ne\-ous.
\end{proposition}
\begin{proof}
	Let $\bB\le\bA^k$ be finite, and let $f:\bB\to\bA$. Let us enumerate the elements of $B$ like $B=\{\tb_1,\ldots, \tb_m\}$. Assume $\tb_i=(b_{i,1},\dots,b_{i,k})$, for $1\le i\le m$. Define $\tau:=\{(b_{1,j},\dots b_{m,j})\mid 1\le j \le k\}$, and $\tc:=(c_1,c_2,\ldots,c_m)$, where $c_i=f(\tb_i)$. Then $\ppTp^{(0)}_{\bA}(\tau)\subseteq \ppTp^{(0)}_{\bA}(\tc)$. By our assumptions, it follows that $\ppTp_{\bA}(\tau)\subseteq \ppTp_{\bA}(\tc)$. 

	Suppose that $\widehat\bB$ is a finite superstructure of $\bB$ in $\bA$. Then we can enumerate the elements of $\widehat{B}$ like $\widehat{B}=\{\tb_1,\ldots,\tb_m,\tb_{m+1},\ldots,\tb_{m+n}\}$. Let $\sigma:=\{(b_{1,j},\dots b_{m+n,j})\mid 1\le j \le k\}$, and let $\Psi:=\ppTp_{\bA}(\sigma)$. Since $\bA$ is pp-atomic, there exists a positive primitive formula $\varphi$ such that $\ppTp_\bA(\varphi^\bA)=\Psi$. Clearly, we have that $(\exists x_{m+1}\ldots\exists x_{m+n} \varphi)\in\ppTp_\bA(\tau)$. Thus also $(\exists x_{m+1}\ldots\exists x_{m+n} \varphi)\in\ppTp_\bA(\tc)$. That means that there exist $c_{m+1},\ldots, c_{m+n}$ in $A$, such that
	\[
		\varphi\in\ppTp_\bA (c_1,\ldots,c_m,c_{m+1},\ldots, c_{m+n}).
	\]
	But from this it follows that $\ppTp_{\bA}\big((c_1,\dots,c_{m+n})\big)\supseteq\Psi=  \ppTp_{\bA}(\sigma)$. Consequently, the mapping $\tb_i\mapsto c_i$ ($1\le i\le m+n$) defines a homomorphism that extends $f$ to $\widehat{\bB}$.

	This shows that $\bA$ is weakly po\-ly\-mor\-phism-ho\-mo\-ge\-ne\-ous.
\end{proof}

\begin{proposition}\label{gentypesPL}
  Let $\mathbf{A}=(A,({\varrho}_{\mathbf{A}})_{\varrho\in R})$ be an algebraic 
  relational structure. If every complete positive primitive  type $\Psi$ over $\bA$ is generated by $\Psi^{(k)}$, then $\mathbf{A}$ is $k$-polylocal.
\end{proposition}
\begin{proof}
	Let $\tau\subseteq A^m$ be finite. Then $\ppTp_\bA(\tau)$ is a complete positive primitive type over $\bA$.  Hence, by the assumption we have that $\ppTp_\bA(\tau)$ is generated by $\ppTp_\bA^{(k)}(\tau)$. In particular, we conclude that $\big(\ppTp_\bA(\tau)\big)^\bA=\big(\ppTp_\bA^{(k)}(\tau)\big)^\bA$. Hence, by Lemma~\ref{mainlemma} we have that $\bA$ is $k$-polylocal.  
\end{proof}

\subsection*{Positive primitive elimination sets}
Let $\bA$ be an $L$-structure, and let $\Phi$ be a set of positive primitive formulae from $L_{\omega\omega}$. We call $\Phi$ a \emph{positive primitive elimination set for $\bA$} if for every positive primitive formula $\varphi(x_1,\dots,x_m)\in L_{\omega\omega}$ there exists a finite set $\Psi\subseteq\Phi$ such that all formulae from $\Psi$ have their free variables in the set $\{x_1,\dots,x_m\}$ and such that $\varphi^\bA =\Psi^\bA$. 

Specifically, we say that $\bA$ has \emph{quantifier elimination for positive primitive formulae} if the set of quantifier free positive primitive formulae from $L_{\omega\omega}$ is a positive primitive elimination set for $\bA$.
\begin{proposition}\label{localQE}
	Let $\bA$ be a weakly oligomorphic $k$-polylocal relational structure. Then the set of all positive primitive formulae of quantifier depth at most $k$ from $L_{\omega\omega}$ is a positive primitive elimination set for $\bA$. 
\end{proposition}
\begin{proof}
	Let $m\in\bN\setminus\{0\}$, and let $\varphi(x_1,\dots,x_m)\in L_{\omega\omega}$ be positive primitive. Take $\varrho:=\varphi^\bA$. Since $\bA$ is weakly oligomorphic, there exists a finite subset $\tau$ of $\varrho$ such that $(\ppTp_\bA(\tau))^\bA=\varrho$. Otherwise we would obtain an infinite properly increasing chain of positive primitively definable relations over $\bA$ whose union is equal to $\varrho$ --- contradiction. By Lemma~\ref{mainlemma} we have $(\ppTp_\bA(\tau))^\bA=(\ppTp_\bA^{(k)}(\tau))^\bA$. Again using that $\bA$ is weakly oligomorphic, there exists a  $\psi\in\ppTp_\bA^{(k)}(\tau)$ such that $\psi^\bA=(\ppTp_\bA^{(k)}(\tau))^\bA=\varrho=\varphi^\bA$. 
\end{proof}   

\begin{proposition}\label{ppesee}
	Let $\bA$ be a relational structure that has a positive primitive elimination set of formulae of quantifier depth at most $k$. Then every closed positive primitive type $\Psi$ of $\bA$ is generated by $\Psi^{(k)}$. 
\end{proposition}
\begin{proof}
Let $\Psi$ be an $m$-ary closed positive primitive type over $\bA$.  For every $\psi\in\Psi$, let $\Psi_\psi$ be a finite set of positive primitive formulae of quantifier depth $\le k$ with free variables in $\{x_1,\dots,x_m\}$ such that $\psi^\bA=(\Psi_\psi)^\bA$.  Define $\Psi^*:=\bigcup_{\psi\in\Psi} \Psi_\psi$. Then, by construction, $\ppTp_\bA(\Psi^\bA)=\ppTp_\bA((\Psi^*)^\bA)$. Moreover, since $\Psi^*\subseteq\Psi^{(k)}$, we have $\ppTp_\bA((\Psi^*)^\bA)\subseteq \ppTp_\bA((\Psi^{(k)})^\bA)\subseteq\ppTp_\bA(\Psi^\bA)=\Psi$. Hence $\ppTp_\bA\big((\Psi^{(k)})^\bA\big)=\Psi$.
\end{proof}

\begin{corollary}\label{corquant}
	A weakly oligomorphic relational structure has quantifier elimination for positive primitive formulae if and only if it is weakly po\-ly\-mor\-phism-ho\-mo\-ge\-ne\-ous. 
\end{corollary}
\begin{proof}
	Let $\bA$ be a weakly oligomorphic relational structure. Then, in particular, $\bA$ is algebraic and pp-atomic.
	
	``$\Rightarrow$'' By Proposition~\ref{ppesee},  every complete positive primitive type $\Psi$ over $\bA$ is generated by $\Psi^{(0)}$. From Proposition~\ref{watwPH} it follows that $\bA$ is weakly po\-ly\-mor\-phism-ho\-mo\-ge\-ne\-ous.
	
	``$\Leftarrow$'' From Proposition~\ref{wPHPL} it follows that $\bA$ is polylocal. Now, by Proposition~\ref{localQE} we have that $\bA$ has quantifier elimination for positive primitive formulae.
\end{proof}

\begin{remark}
	A similar result for $\omega$-categorical structures was obtained By B.~Romov in \cite[Thm.3.4]{Rom11}. According to Romov, the proof of \cite[Thm.3.4]{Rom11}  generalizes to a proof of Corollary~\ref{corquant}, restricted to countable structures. 
\end{remark}

\subsection*{Collected findings}
The results of this section can be visualized by two diagrams:
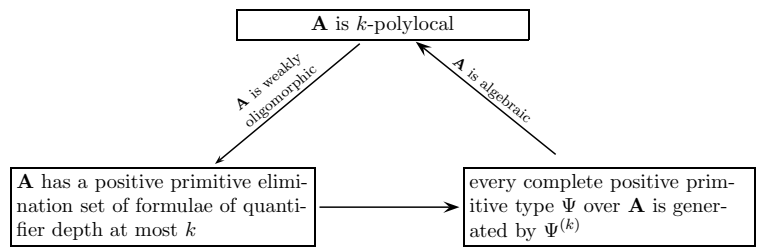
\begin{center}\psscalebox{0.7 0.7}{
   	\begin{pspicture}(0,-1)(14,4)
		\rput[t](7,3.5){\rnode{EL}{\fbox{\parbox{5.25cm}{\centerline{$\mathbf{A}$ is $k$-polylocal}}}}}
		\rput[tl](0,0.5){\rnode{PEES}{\fbox{\parbox{5.5cm}{\RaggedRight $\bA$ has a positive primitive elimination set of formulae of quantifier depth at most $k$}}}}
		\rput[tr](14,0.5){\rnode{TYPES}{\fbox{\parbox{5.25cm}{\RaggedRight every complete positive primitive type $\Psi$ over $\bA$ is generated by $\Psi^{(k)}$}}}}
		{\psset{arrowscale=2,nodesep=4pt}
		\ncline[offset=4pt]{<-}{EL}{TYPES}\naput[nrot=:U]{\scriptsize $\mathbf{A}$ is algebraic}}
		\ncline{<-}{PEES}{EL}\naput[nrot=:U]{\scriptsize\shortstack{$\mathbf{A}$ is weakly\\ oligomorphic}}
		\ncline[arrowscale=2]{->}{PEES}{TYPES}
		\end{pspicture}}
	\end{center}

\begin{center}\psscalebox{0.7 0.7}
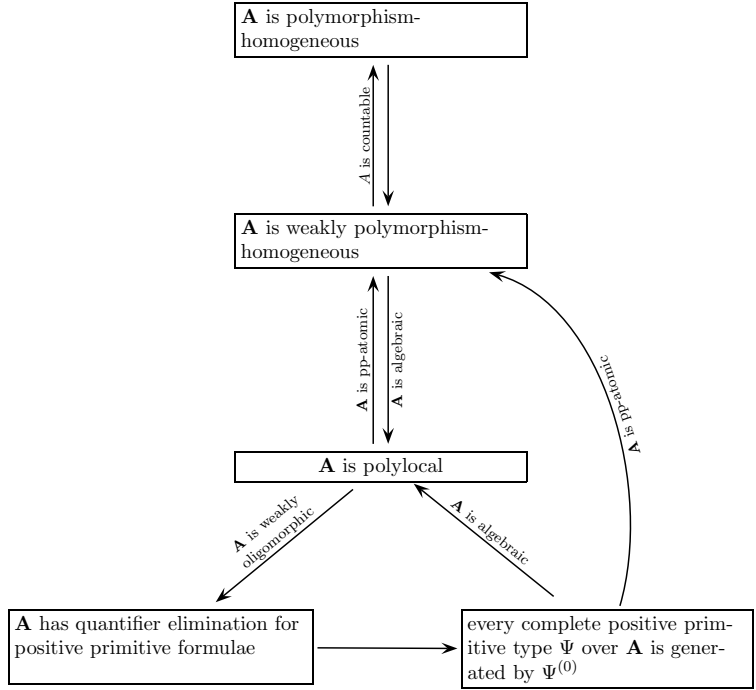
{
   	\begin{pspicture}(0,-1)(14,13)
		\rput[t](7,12){\rnode{HH}{\fbox{\parbox{5.25cm}{\RaggedRight $\mathbf{A}$ is polymorphism-homogeneous}}}}
		\rput[t](7,8){\rnode{OPE}{\fbox{\parbox{5.25cm}{\RaggedRight $\mathbf{A}$ is weakly polymorphism-homogeneous}}}}
		\rput[t](7,3.5){\rnode{EL}{\fbox{\parbox{5.25cm}{\centerline{$\mathbf{A}$ is polylocal}}}}}
		\rput[tl](0,0.5){\rnode{PEES}{\fbox{\parbox{5.5cm}{\RaggedRight $\bA$ has quantifier elimination for positive primitive formulae\\~}}}}
		\rput[tr](14,0.5){\rnode{TYPES}{\fbox{\parbox{5.25cm}{\RaggedRight every complete positive primitive type $\Psi$ over $\bA$ is generated by $\Psi^{(0)}$}}}}
		{\psset{arrowscale=2,nodesep=4pt}
		\ncline[offset=4pt]{<-}{EL}{TYPES}\naput[nrot=:U]{\scriptsize $\mathbf{A}$ is algebraic}}
		{\psset{arrowscale=2,nodesep=4pt}
		\ncline{<-}{PEES}{EL}\naput[nrot=:U]{\scriptsize\shortstack{$\mathbf{A}$ is weakly\\ oligomorphic}}
		\ncline[offset=4pt]{->}{EL}{OPE}\naput[nrot=:U]{\scriptsize $\mathbf{A}$ is pp-atomic}
		\ncline[offset=-4pt]{<-}{EL}{OPE}\nbput[nrot=:U]{\scriptsize $\mathbf{A}$ is algebraic}
		\ncline[offset=4pt]{->}{OPE}{HH}\naput[nrot=:U]{\scriptsize $A$ is countable} \ncline[offset=-4pt]{<-}{OPE}{HH}}
		\ncarc[arrowscale=2,arcangleA=-45,arcangleB=-45]{->}{TYPES}{OPE}\nbput[nrot=:U]{\scriptsize $\mathbf{A}$ is pp-atomic}
			\ncline[arrowscale=2]{->}{PEES}{TYPES}
		\end{pspicture}}
	\end{center}

\section{Characterizing positive primitively definable relations}\label{lcRC}

In the following, with $R_A^{(m)}$ we will denote the set of all $m$-ary relations on $A$. Moreover, we define
\[R_A:=\bigcup_{m\in\bN\setminus\{0\}} R_A^{(m)}.\]
A relational algebra $W$ on a set $A$ is a subset of $R_A$ with the property that 
\[ W=[\bA]_\RA\]
for some relational structure $\bA$ on $A$. Note that a set of relations on $A$ is a relational algebra if and only if it is closed with respect to positive primitive definitions, hence the name relational algebra. 

Given any set $W$ of relations on $A$, we can construct a \emph{canonical relational structure} $\bC_W$:  We take as relational signature $L$ the set $W$ itself. The arity of each element of $L$ is the arity of this element, considered as a relation. Moreover, for every $\varrho\in L$ we define  $\varrho_{\bC_W}:=\varrho$. It is easy to see that $W$ is a relational algebra if and only if $W=[\bC_W]_\RA$. Moreover, in general $[\bC_W]_\RA$ is the smallest relational algebra that contains $W$. We will write  $[W]_\RA$ instead of $[\bC_W]_\RA$. Moreover, instead of $\Pol(\bC_W)$ we will write  $\Pol(W)$.

As usual, we define $O_A^{(n)}$ to be the set of all functions from $A^n$ to $A$, and we set 
\[O_A:=\bigcup_{n\in\bN\setminus\{0\}} O_A^{(n)}.\]

If $f\in O_A^{(n)}$, and $\varrho\subseteq A^m$ we call $\varrho$ \emph{invariant} for $f$ (and we write $f\pres\varrho$ whenever 
$f$ is a polymorphism of $\bC_{\{\varrho\}}$). In other words, $f\pres\varrho$ if and only if for all $\ta_1,\ldots,\ta_m\in A^m$ with $\ta_i=(a_{i,1},\ldots,a_{i,n})$ holds that
\[ 
	\begin{bmatrix}
		a_{1,1} \\
		\vdots \\
		a_{m,1} \\
	\end{bmatrix}\in\varrho,\ldots,
	\begin{bmatrix}
		a_{1,n} \\
		\vdots \\
		a_{m,n} \\
	\end{bmatrix}\in\varrho \Longrightarrow
	\begin{bmatrix}
		f(a_{1,1},\ldots,a_{1,n}) \\
		\vdots \\
		f(a_{m,1},\ldots,a_{m,n}) \\
	\end{bmatrix}\in\varrho.\]
For any subset $F\subseteq O_A$ we define 
\[\Inv(F):=\{\varrho\in R_A\mid \forall f\in F: f\pres\varrho\},\]
and for any $m\in\bN\setminus\{0\}$, we define $\Inv^{(m)}(F):=\Inv(F)\cap R_A^{(m)}$. 
The operators $\Inv$ and $\Pol$ form a Galois-connection between sets of functions and sets of relations on $A$. It is well known that every Galois-closed set of relations forms a relational algebra.  
\begin{lemma}\label{algclonecomp}
	Let $\bA$ be a relational structure. Then
	\[[\bA]_\RA\subseteq \overline{[\bA]_\RA}\subseteq \Inv(\Pol(\bA)).\]
\end{lemma}
\begin{proof}
	It is folklore that every relation that can be defined by a positive primitive formula over $\bA$ is invariant under all polymorphisms of $\bA$. However, this just means that we generally have $[\bA]_\RA\subseteq\Inv(\Pol(\bA))$. Note that for every $m\in\bN\setminus\{0\}$ we have that $\overline{[\bA]_\RA^{(m)}}$ is the closure of $[\bA]_\RA^{(m)}$ under arbitrary intersections. Since $\Inv^{(m)}(\Pol(\bA))$ is closed with respect to arbitrary intersections, it follows that $\overline{[\bA]_\RA^{(m)}}\subseteq \Inv^{(m)}(\Pol(\bA))$.  We conclude that $\overline{[\bA]_\RA}\subseteq\Inv(\Pol(\bA))$. 
\end{proof}
In general the Galois-closed sets of relations are called \emph{relational clones}. It is one of the fundamental theorems of general algebra that whenever $A$ is a finite set then the relational clones on $A$ and the relational algebras on $A$ coincide. However, when $A$ is infinite, then there exist relational algebras that are not relational clones. A relational algebra $W$ is a relational clone if and only if $W=\Inv(\Pol(W))$. 

A \emph{clone} on $A$ is a subset of $O_A$ that contains all projections and that is closed with respect to composition. Every Galois closed set of functions is a clone. In general, Galois-closed sets of functions on $A$ are called \emph{local clones}.  Another fundamental theorem of general algebra states that on a finite set $A$ all clones are local. However, when $A$ is infinite, then there are clones that are not local. Standard references for clones on finite sets are \cite{PoeKal79,Sze86,Lau06}. For a survey of known results about clones on infinite sets we refer to \cite{GolPin08}. 

If for some relational structure $\bA$ we have $[\bA]_\RA=\Inv(\Pol(\bA))$, then this means that a relation $\varrho$ on $A$ is positive primitively definable over $\bA$ if and only if it is invariant for all polymorphisms of $\bA$. In the following it is our goal to characterize all those relational structures with this property. 
\begin{lemma}\label{invpolalg}
	Let $\bA$ be a relational structure and let $m\in\bN\setminus\{0\}$. Then the closure system $\Inv^{(m)}(\Pol(\bA))$ is algebraic. 
\end{lemma}
\begin{proof}
	Consider the algebra $\mathbb{A}$ with carrier $A$ that has as basic operations all elements of $\Pol(\bA)$. We will show:
	\begin{equation}\label{eq4}
		\forall\sigma\subseteq A^m : \sigma\in\Inv^{(m)}(\Pol(\bA)) \iff \sigma \text{ induces  a subalgebra of } \mathbb{A}^m.
	\end{equation}
	``$\Rightarrow$'' Let $\sigma\in\Inv^{(m)}(\Pol(\bA))$. Let $f^{\mathbb{A}}$ be a basic operation of $\mathbb{A}$ --- say, of arity $n$. Then, by definition, $f^{\mathbb{A}}\in\Pol^{(n)}(\bA)$. Let $f^{\mathbb{A}^m}$ be the corresponding operation of $\mathbb{A}^m$. Let $\ta_1,\dots,\ta_n\in \sigma$ with $\ta_i=(a_{i,1},\dots,a_{i,m})$ for $1\le i\le n$. Then we have
	\begin{equation}\label{eq5}
	 f^{\mathbb{A}^m}(\ta_1,\dots,\ta_n) = \big(f^{\mathbb{A}}(a_{1,1},\dots,a_{n,1}),\dots,f^{\mathbb{A}}(a_{1,m},\dots,a_{n,m})\big)=:\tb. 
	\end{equation}
	Since $f^{\mathbb{A}}\in\Pol(\bA)$, it follows that $\tb\in\sigma$. Thus, $\sigma$ is a universe of a subalgebra of $\mathbb{A}$.   
	
	``$\Leftarrow$'' Let $\sigma$ be the universe of a subalgebra of $\mathbb{A}^m$, and let $g\in\Pol^{(n)}(\bA)$. Then $g$ is equal to a basic operation $f^{\mathbb{A}}$ of $\mathbb{A}$. Let $f^{\mathbb{A}^m}$ be the corresponding basic operation of $\mathbb{A}^m$. Finally, let $\ta_1,\dots,\ta_n\in \sigma$ with $\ta_i=(a_{i,1},\dots,a_{i,m})$ for $1\le i\le n$. Then, by \eqref{eq5}, we have 
	\[ \big(g(a_{1,1},\dots,a_{n,1}),\dots,g(a_{1,m},\dots,a_{n,m})\big)= f^{\mathbb{A}^m}(\ta_1,\dots,\ta_n)=:\tb.\]
	In particular, we conclude that $\tb\in\sigma$. Thus, $g$ preserves $\sigma$, whence $\sigma\in\Inv^{(m)}(\Pol(\bA))$. 
	
	Thus \eqref{eq4} is proved. From the definition of algebraic closure systems, it follows that $\Inv^{(m)}(\Pol(\bA))$ is algebraic. 
\end{proof}

\begin{theorem}\label{odrugom}
	Let $\bA$ be a countable algebraic  po\-ly\-mor\-phism-ho\-mo\-ge\-ne\-ous relational structure, that has principal positive primitive types. Then  $[\bA]_{\RA}=\Inv(\Pol (\bA))$.
\end{theorem}
\begin{proof}
	We will show that for every $m\in \bN\setminus\{0\}$
	\[
	\Inv^{(m)}(\Pol (\bA)) =[\bA]^{(m)}_{\RA}.
	\]
	From Lemma~\ref{algclonecomp} we have that $[\bA]^{(m)}_{\RA}\subseteq \Inv^{(m)}(\Pol (\bA))$. So let $\sigma\in\Inv^{(m)}(\Pol(\bA))$. By Lemma~\ref{invpolalg},  we have that $\Inv^{(m)}(\Pol(\bA))$ is an algebraic closure system. Let us denote the closure operator of this closure system with $\Gamma_{\Pol(\bA)}$. Then, by Theorem~\ref{schmidt}, we have
	\begin{equation}\label{s}
		\sigma=\bigcup_{\substack{\tau\subseteq\sigma\\\tau\text{ finite}}} \Gamma_{\Pol(\bA)}(\tau).
	\end{equation}
	We claim now that 
	\begin{equation}
		\forall\tau\subseteq A^m\text{ finite }: \big(\ppTp_\bA(\tau)\big)^\bA= \Gamma_{\Pol(\bA)}(\tau).\label{ss}
	\end{equation}
	So let $\tau\subseteq A^m$ be finite. Then we have that $\big(\ppTp_\bA(\tau)\big)^\bA\in\overline{[\bA]_\RA^{(m)}}$. By Lemma~\ref{algclonecomp}, we have $\overline{[\bA]_\RA^{(m)}}\subseteq\Inv^{(m)}(\Pol(\bA))$. Thus, we have $\Gamma_{\Pol(\bA)}(\tau)\subseteq\big(\ppTp_\bA(\tau)\big)^\bA$. To show the other inclusion, let $\tb\in\big(\ppTp_\bA(\tau)\big)^\bA$. Suppose that $\tau=\{\ta_1,\dots ,\ta_n\}$, with $\ta_i=(a_{i,1},\dots,a_{i,m})$, and let $\bB$ be the substructure of $\bA^n$ induced by $\{(a_{1,j},\dots, a_{n,j})\mid 1\le j\le m\}$ --- that is, we have the following situation:
	\[
	\begin{array}{c}
	\begin{bmatrix}
		a_{1,1} & a_{2,1} &\dots & \Rnode{an1}{a_{n,1}}\\
		a_{1,2} & a_{2,2} &\dots & \Rnode{an2}{a_{n,2}}\\
		\vdots & \vdots &\ddots & \vdots\\
		\Rnode{a1m}{a_{1,m}} & \Rnode{a2m}{a_{2,m}} &\dots & \Rnode{anm}{a_{n,m}}
		\uput{2ex}[-90]{-90}({a1m}){\in\tau}
		\uput{2ex}[-90]{-90}({a2m}){\in\tau}
		\uput{2ex}[-90]{-90}({anm}){\in\tau}
		\uput{4ex}[0]{0}({an1}){\in\bB}
		\uput{4ex}[0]{0}({an2}){\in\bB}
		\uput{4ex}[0]{0}({anm}){\in\bB.}
	\end{bmatrix}\\\vspace{2ex}
	\end{array}
	\]
	Since  $\ppTp_\bA(\tau)\subseteq\ppTp_\bA(\tb)$, it follows that the mapping $f:\bB\rightarrow \bA$ defined by $f(a_{1,j}, \dots, a_{n,j}):=b_j$, for $j\in \{1,\dots ,m\}$ defines a local polymorphism of $\bA$. Since $\mathbf{A}$ is po\-ly\-mor\-phism-ho\-mo\-ge\-ne\-ous, it follows that $f$ can be extended to a polymorphism $\hat{f}$ of $\bA$. But, from this it follows at once that $\bar{b}\in \Gamma_{\Pol (\bA)} (\tau)$. Thus we proved \eqref{ss}.

	Now, using \eqref{s} together with \eqref{ss}, we obtain
	\begin{equation}\label{neweq}
	\sigma=\bigcup_{\substack{\tau\subseteq\sigma\\\tau\text{ finite}}} \big(\ppTp_\bA(\tau)\big)^\bA.
	\end{equation}

	Since $\bA$ is algebraic, from  \eqref{neweq} together with Theorem~\ref{schmidt} it follows that $\sigma\in \overline{[\bA]_\RA^{(m)}}$. Finally, since $\bA$ has principal positive primitive types, from Lemma~\ref{princint} we conclude that $\overline{[\bA]_\RA^{(m)}}= [\bA]_\RA^{(m)}$. Thus $\sigma\in [\bA]_\RA^{(m)}$.
\end{proof}

\begin{theorem}\label{locclosedStruc}
	Let $\bA$ be a countable relational structure. Then $[\bA]_{\RA}=\Inv(\Pol (\bA))$ if and only if $\bA$ is algebraic and has principal positive primitive types.
\end{theorem}
\begin{proof}
	($\Leftarrow$) Let $\widehat{\bA}$ be the canonical structure of $[\bA]_\RA$. Clearly, then $[\widehat\bA]_\RA=[\bA]_\RA$, and $\widehat\bA$ has quantifier elimination for positive primitive formulae. Hence, by Proposition~\ref{ppesee}, all closed positive primitive types over $\widehat\bA$ are generated by their quantifier free part. Clearly, $\widehat{\bA}$ is also algebraic and has principal positive primitive types.

	From Proposition~\ref{watwPH} it follows that $\widehat\bA$ is weakly po\-ly\-mor\-phism-ho\-mo\-ge\-ne\-ous. As $A$ is countable, it follows that $\widehat\bA$ is po\-ly\-mor\-phism-ho\-mo\-ge\-ne\-ous. Finally, from Theorem~\ref{odrugom} we obtain that
	\[[\bA]_\RA=[\widehat\bA]_\RA=\Inv(\Pol(\widehat\bA)) = \Inv(\Pol(\bA)).\]

	\noindent  ($\Rightarrow$) In general, we have
	\[[\bA]_\RA\subseteq \overline{[\bA]_\RA}\subseteq\Inv(\Pol(\bA)).\]
	So, if $[\bA]_\RA=\Inv(\Pol(\bA))$, then we have in particular, that $[\bA]_\RA=\overline{[\bA]_\RA}$. Hence, by Lemma~\ref{princint}, we have that $\bA$ has principal positive primitive types. 

	It is easy to see, that $\Inv(\Pol(\bA))$ is always closed with respect to arbitrary intersections and unions of upward directed sets of relations. Again using that $[\bA]_\RA=\overline{[\bA]_\RA}=\Inv(\Pol(\bA))$, we conclude with Theorem~\ref{schmidt}, that $\bA$ is algebraic. 
\end{proof}

\begin{remark}
	According to \cite{Rom11}, a structure $\bA$ with the property that $[\bA]_{\RA}=\Inv(\Pol (\bA))$ is also called a \emph{positive primitive structure}. It was proved in \cite[Thm.4]{BodNes06}, that every countable $\omega$-categorical structure is positive primitive (cf. also \cite[Prop.2.4]{Rom11}). Moreover, it was shown in \cite[Thm.2.5]{Rom11}, that every polymorphism-homogeneous relational structure over a signature with just one relation is positive primitive. Theorem~\ref{locclosedStruc} generalizes these results in that it gives a complete characterization of countable positive primitive structures. An immediate consequence of our result is that every weakly oligomorphic structure is positive primitive and hence, that every polymorphism-homogeneous structure over a finite signature is positive primitive, too. 
\end{remark}

\begin{corollary}[{Romov \cite[Thm.3.5]{Rom08}}]\label{locclosedRC}
  Let $W$ be a relational algebra on a countable set. Then $W$ is a relational clone (i.e.\ $W=\Inv(\Pol(W))$) if and only if $W$ is closed with respect to unions of upwards directed sets of relations and with respect to arbitrary intersections of relations of equal arities.
\end{corollary}
\begin{proof}
	($\Leftarrow$) Let $\bC_W$ be the canonical structure of $W$. Then $W=[\bC_W]_\RA$.	Since $W$ is closed with respect to unions of upwards directed sets of relations, by Theorem~\ref{schmidt} we obtain that $\bC_W$ is algebraic. On the other hand, $W$ is also closed with respect to arbitrary intersections, so by Lemma~\ref{princint} it follows that $\bC_W$ has principal positive primitive types.	Hence, by Theorem~\ref{locclosedStruc} we obtain that
	\[W=[\bC_W]_\RA=\Inv(\Pol(\bC_W)) = \Inv(\Pol(W)).\]
\noindent  ($\Rightarrow$) Clear.
\end{proof}

\subsection*{Weak oligomorphy revisited}
A set of relations $W\subseteq R_A$ on a set $A$ will be called \emph{weakly oligomorphic} if $W^{(m)}$ is finite, for all $m\in\bN\setminus\{0\}$. Clearly, a structure $\bA$ is weakly oligomorphic if and only if $[\bA]_\RA$ is weakly oligomorphic. 
\begin{proposition}\label{wo1}
	Let $\bA$ be a relational structure. Then the following are equivalent:
	\begin{enumerate}
		\item $\overline{[\bA]_\RA}$ is weakly oligomorphic,
		\item $[\bA]_\RA$ is weakly oligomorphic,
		\item $[\bA]_\WKA$ is weakly oligomorphic.
	\end{enumerate}
	Here $[\bA]_\WKA$ denotes the relational algebra that consists of all relations on $A$ that are definable by positive existential formulae over $\bA$.
\end{proposition} 
\begin{proof}
	$(1)\Rightarrow (2)$: For every $m$, we have $[\bA]_\RA^{(m)}\subseteq\overline{[\bA]_\RA^{(m)}}$. Hence, $[\bA]_\RA^{(m)}$ is finite, for every $m$.
	
	$(2)\Rightarrow (3)$: For every $m$, $[\bA]_\WKA^{(m)}$ is obtained from $[\bA]_\RA^{(m)}$ through closure with respect to finite unions. From finitely many relations, only finitely many finite unions can be formed. 
	
	$(3)\Rightarrow (1)$: For every $m$, $\overline{[\bA]_\RA^{(m)}}$  is obtained from $[\bA]^{(m)}_\RA$ through closure with respect to arbitrary intersections. If $[\bA]_\WKA^{(m)}$ is finite then so is its subset $[\bA]_\RA^{(m)}$. However, from finitely many relations, only finitely many intersections can be formed.
\end{proof}

When working over a countable basic set $A$, more can be said: 
\begin{proposition}[{Ma\v{s}ulovi\'c \cite{MasWOC}}]\label{wo2}
	Let $\bA$ be a relational structure over a countable basic set $A$. Then the following are equivalent:
	\begin{enumerate}
		\item $\bA$ is weakly oligomorphic,
		\item $\Inv(\End(\bA))$ is weakly oligomorphic,
		\item $\Inv(\Pol(\bA))$ is weakly oligomorphic.
		\item $\Inv(\Pol^{(k)}(\bA))$ is weakly oligomorphic, for some $k\ge 1$
		\item $\Inv(\Pol^{(k)}(\bA))$ is weakly oligomorphic, for all $k\ge 1$
	\end{enumerate}
\end{proposition}
\begin{proof}
	$(1)\iff (2)$: By Proposition~\ref{wo1}, $[\bA]_\WKA$ is weakly oligomorphic. Now the claim follows from \cite[Thm.6.15]{AUpaper}. 

	$(2)\Rightarrow (5)$: For every $f\in\End(\bA)$ we may define an $\hat{f}\in\Pol^{(k)}(\bA)$ through $\hat{f}(x_1,\dots,x_k):=f(x_1)$. Clearly,
	\[ \Inv(\End(\bA))=\Inv(\{\hat{f}\mid f\in\End(\bA)\}).\]
	Hence $\Inv(\Pol^{(k)}(\bA))\subseteq\Inv(\End(\bA))$. So, if $\Inv(\End(\bA))$ is weakly oligomorphic, then so is $\Inv(\Pol^{(k)}(\bA))$. Since $k$ was arbitrary, the claim follows.

	$(5)\Rightarrow (4)$ clear. 

	$(4)\Rightarrow (3)$: Since $\Pol^{(k)}(\bA)\subseteq \Pol(\bA)$, it follows that $\Inv(\Pol(\bA))\subseteq\Inv(\Pol^{(k)}(\bA))$. Hence, if $\Inv(\Pol^{(k)}(\bA))$ is weakly oligomorphic, then so is $\Inv(\Pol(\bA))$. 

	$(3)\Rightarrow (1)$: For every $m$, the elements of $[\bA]_\RA^{(m)}$ are invariant under the polymorphisms of $\bA$. That is, $[\bA]_\RA\subseteq \Inv(\Pol(\bA))$. Hence, if $\Inv(\Pol(\bA))$ is weakly oligomorphic, then so is $[\bA]_\RA$.
\end{proof}
\newcommand{\IC}{\operatorname{IC}}
\section{Polymorphism-homogeneity and the Baker-Pixley theorem}\label{sec5}

In this section, using the Baker-Pixley Theorem, we will prove that po\-ly\-mor\-phism-ho\-mo\-ge\-ne\-i\-ty for finite relational structures is decidable. 

\subsection*{Near unanimity functions}
Let $k\ge 3$. A $k$-ary function $f$ on a set $A$ is called \emph{near unanimity-function} if it fulfills the following, so called, \emph{near-unanimity-identities}:
\[ \forall x,y\,:\,f(y,x,\ldots,x)=f(x,y,\ldots,x)=\cdots = f(x,x,\ldots,y)=x.\]

\begin{proposition}\label{ICNU}
	Let $\bA=\str{A}{L}$ be a relational structure such that the arity of the relational symbols in $L$ is bounded above by $m\ge 2$. If every partial $(m+1)$-ary polymorphism of $\bA$ can be extended to a global polymorphism, then $\Pol(\bA)$ contains an $(m+1)$-ary near unanimity-function.
\end{proposition}
\begin{proof}
	We define $f: A^{m+1}\prtto A$ in the following way: For all $x,y\in A$ we define
	\[ f(y,x,\ldots,x)=f(x,y,\ldots,x)=\cdots =f(x,x,\ldots,y):=x.\]
	Our goal in the following will be to show, that the thus defined partial function is in fact a partial polymorphism of the given structure $\bA$:

	Take an arbitrary $\varrho\in L$ and $\ta_1,\ldots,\ta_{\ar(\varrho)}\in\dom f$ (where $\ta_i=(a_{i,1},\ldots,a_{i,m+1})$), such that
	\[
	\begin{matrix}
		a_{1,1}            & a_{1,2}            & \cdots & a_{1,m+1}            & \in \dom f\\
		a_{2,1}            & a_{2,2}            & \cdots & a_{2,m+1}            & \in \dom f\\
		\vdots             & \vdots             & \ddots & \vdots               & \vdots\\
		a_{\ar(\varrho),1} & a_{\ar(\varrho),2} & \cdots & a_{\ar(\varrho),m+1} &  \in \dom f\\
		\rput{-90}{\in}    & \rput{-90}{\in}    & \cdots & \rput{-90}{\in}      & \\
		\varrho_\bA        & \varrho_\bA        & \cdots & \varrho_\bA          &
	\end{matrix}
	\] 
	The key is that this matrix has exactly $m+1$ columns but at most $m$ rows. Since every row is in the domain of $f$, in each row there is at most one element that occurs exactly once in this row. These elements of the matrix we shall call \emph{distinguished}. Altogether the matrix contains at most $m$ distinguished elements. Therefore there has to be at least one column that contains not a single distinguished entry. However, this column must be equal to
	\[ 
	\begin{bmatrix}
		f(a_{1,1},a_{1,2},\ldots,a_{1,m+1})\\
		f(a_{2,1},a_{2,2},\ldots,a_{2,m+1})\\
		\vdots\\
		f(a_{\ar(\varrho),1},a_{\ar(\varrho),2},\ldots,a_{\ar(\varrho),m+1})
	\end{bmatrix}
	\]
	whence $f$ is a partial polymorphism. Hence, by our assumptions on $\bA$, we conclude that $f$ can be extended to a polymorphism $g$ of $\bA$. This function $g$ is a near unanimity function.
\end{proof}
\begin{corollary}
	If $\bA$ is a finite po\-ly\-mor\-phism-ho\-mo\-ge\-ne\-ous relational structure over a finite relational signature, then $\Pol(\bA)$ contains a near unanimity function.\qed
\end{corollary}

The existence of a near unanimity polymorphism in a structure $\bA$ has strong consequences for the clone of polymorphisms of $\bA$. Let $\mathbb{A}$ be an algebra with basic set $A$ and with all functions from $\Pol(\bA)$ as fundamental operations. Let further $V(\mathbb{A})$ be the variety generated by $\mathbb{A}$. If $\bA$ has a near unanimity polymorphism $f$, then there exists a term $t$ in the language of $\mathbb{A}$, such that the term function defined by $t$ in any algebra from $V(\mathbb{A})$ is a near unanimity function. Thus, the Baker-Pixley theorem can be invoked for $V(\mathbb{A})$:
\begin{theorem}[Baker, Pixley \cite{BakPix75}]
	Let $V$ be a variety and $d\ge 2$ be an integer. Then the following are equivalent:
	\begin{enumerate}
		\item There is a $(d+1)$-ary term $t$ in the language of $V$ such that the term function of $t$ in any algebra from $V$ is a near unanimity function,
		\item if $\mathbb{A}$ is a subalgebra of a direct product $\mathbb{P}=\mathbb{B}_1\times\dots\times\mathbb{B}_r$ for $r\ge d$, then $\mathbb{A}$ is determined by all its projections onto $d$ coordinates of this product,
		\item if $\mathbb{A}\in V$ and if from $r$ congruences $x\equiv a_i\pmod{ \theta_i}$ all collections of $d$ congruences are solvable, then all $r$ congruences are solvable, simultaneously,
		\item if $\mathbb{A}\in V$, $n\ge 1$, and if $f:A^n\prtto A$ is a partial function with a finite domain, then $f$ extends to a term-function of $\mathbb{A}$ if and only if every restriction of $f$ to $d$ or fewer elements of its domain extends to a term function of $\mathbb{A}$,
		\item if $\mathbb{A}\in V$, $n\ge 1$, and if $f:A^n\prtto A$ is a partial function with a finite domain, then $f$ extends to a term function of $\mathbb{A}$ if and only if $f$ preserves all relations from $\Inv^{(d)}(\mathbb{A})$.  	   
	\end{enumerate}
\end{theorem}

\begin{theorem}
	Let $\bA$ be a finite relational structure, all of whose relations have arity $\le d$ for some $d\ge 2$. If $\bA$ is $|A|^d$-po\-ly\-mor\-phism-ho\-mo\-ge\-ne\-ous, then it is po\-ly\-mor\-phism-ho\-mo\-ge\-ne\-ous. 
\end{theorem}
\begin{proof}
	Without loss of generality, we may assume that the relational signature of $\bA$ is finite.
	
	If $|A|=1$, then the claim is trivially true. Therefore in the following we will assume that $A$ has at least $2$ elements. With this assumption we always have $|A|^d> d+1$. By Proposition~\ref{ICNU} we have that $\bA$ has a $(d+1)$-ary near unanimity polymorphism. 
	
	Let $f$ be an $n$-ary partial polymorphism of $\bA$ such that $n$ is larger than $|A|^d$ and such that the  domain of $f$ has $m$ elements for some $1\le m\le d$. Suppose $\dom(f)=\{(a_{i,1},\dots, a_{i,n})\mid 1\le i\le m\}$. Then the matrix
	\[
	\begin{bmatrix}
		a_{1,1} & a_{1,2}& \dots & a_{1,n}\\
		a_{2,1} & a_{2,2}& \dots & a_{2,n}\\
		\vdots & \vdots & \ddots & \vdots \\
		a_{m,1} & a_{m,2}& \dots & a_{m,n}
	\end{bmatrix}
	\]
	has at most $|A|^m$ different columns. After removing all duplicate columns we obtain a matrix	
	\[
	\begin{bmatrix}
		a_{1,j_1} & a_{1,j_2}& \dots & a_{1,j_k}\\
		a_{2,j_1} & a_{2,j_2}& \dots & a_{2,j_k}\\
		\vdots & \vdots & \ddots & \vdots \\
		a_{m,j_1} & a_{m,j_2}& \dots & a_{m,j_k}
	\end{bmatrix}
	\]
	Now, the partial function $g:A^k\prtto A$, defined by
	\[
		g(a_{i,j_1},a_{i,j_2},\dots,a_{i,j_k}):=f(a_{i,1},a_{i,2},\dots, a_{i,n}) \quad 1\le i\le m
	\]
	is a partial polymorphism of $\bA$. Since $k\le |A|^d$, and since $\bA$ is $|A|^d$-po\-ly\-mor\-phism-ho\-mo\-ge\-ne\-ous, $g$ extends to a polymorphism $\hat{g}$ of $\bA$. Now we define $\hat{f}:A^n\to A$ according to
	\[\hat{f}(x_1,\dots,x_n):=\hat{g}(x_{i_1},\dots, x_{i_k}).\]
	Clearly, $\hat{f}$ is a polymorphism of $\bA$ that extends $f$. 
	
	Hence, from the Baker-Pixley Theorem it follows that $\bA$ is po\-ly\-mor\-phism-ho\-mo\-ge\-ne\-ous.
\end{proof}

\begin{corollary}
	It is decidable whether a finite relational structure is po\-ly\-mor\-phism-ho\-mo\-ge\-ne\-ous.\qed
\end{corollary}

\section{Classifying po\-ly\-mor\-phism-ho\-mo\-ge\-ne\-ous structures}\label{sec6}
Until now our knowledge about po\-ly\-mor\-phism-ho\-mo\-ge\-ne\-ous structures is mostly theoretical. In this section we are going to change this situation by giving a complete classifications of po\-ly\-mor\-phism-ho\-mo\-ge\-ne\-ous graphs, posets, and strict posets. Finally we will recall Kaarli's classification of po\-ly\-mor\-phism-ho\-mo\-ge\-ne\-ous meet-complete lattices of equivalence relations \cite{Kaa83}. Our results partially depend on  classification results of ho\-mo\-mor\-phism-ho\-mo\-ge\-ne\-ous structures by Cameron, Lockett, Ma\v{s}ulovic, and Ne\v{s}et\v{r}il \cite{CamLoc10,Mas07,CamNes06}.

\subsection*{Polymorphism-homogeneous graphs}
Ho\-mo\-mor\-phism-ho\-mo\-ge\-ne\-ous structures were first searched among graphs and posets in \cite{CamNes06}. While the characterization of ho\-mo\-mor\-phism-ho\-mo\-ge\-ne\-ous posets was meanwhile completed, to our knowledge such a characterization for ho\-mo\-mor\-phism-ho\-mo\-ge\-ne\-ous graphs  is still to be given. What we know from \cite{CamNes06} is that every countable graph that contains the Rado-graph as a spanning subgraph is ho\-mo\-mor\-phism-ho\-mo\-ge\-ne\-ous, and that the disconnected ho\-mo\-mor\-phism-ho\-mo\-ge\-ne\-ous graphs are exactly the disjoint unions of complete graphs of the same size. From \cite{RusSch10}, we know that there are countable ho\-mo\-mor\-phism-ho\-mo\-ge\-ne\-ous graphs that do not belong to the former classes. In this section we are going to give a complete classification of the countable po\-ly\-mor\-phism-ho\-mo\-ge\-ne\-ous graphs.

When we talk about graphs, we mean simple graphs. In particular for us a graph is a pair $(V,\varrho)$, where $V$ is a set of vertices and where $\varrho$ is a symmetric, irreflexive binary relation on $V$. Thus, for us, homomorphisms are not allowed to contract edges to loops.

Recall that for every $k\in\bN\setminus\{0\}$, the star graph $S_k$ is defined to be the complete bipartite graph $K_{1,k}$.
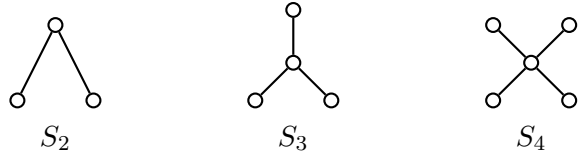
\begin{center}
	\begin{pspicture}(-1,-1)(2,1.5)
		\cnode(0,0){3pt}{a1}
		\cnode(1,0){3pt}{a2}
		\cnode(0.5,1){3pt}{a3}
		\psset{nodesep=0pt}
		\ncline{a1}{a3}
		\ncline{a2}{a3}
		\rput(0.5,-0.5){$S_2$}
	\end{pspicture}
	\begin{pspicture}(-1,-1)(2,1)
		\cnode(0,0){3pt}{a1}
		\cnode(1,0){3pt}{a2}
		\cnode(0.5,1.2){3pt}{a3}
		\cnode(0.5,0.5){3pt}{c}
		\psset{nodesep=0pt}
		\ncline{a1}{c}
		\ncline{c}{a3}
		\ncline{a2}{c}
		\rput(0.5,-0.5){$S_3$}
	\end{pspicture}
	\begin{pspicture}(-1,-1)(2,1)
		\cnode(0,0){3pt}{a1}
		\cnode(1,0){3pt}{a2}
		\cnode(0,1){3pt}{a3}
		\cnode(1,1){3pt}{a4}
		\cnode(0.5,0.5){3pt}{c}
		\psset{nodesep=0pt}
		\ncline{a1}{c}
		\ncline{c}{a4}
		\ncline{c}{a3}
		\ncline{a2}{c}
		\rput(0.5,-0.5){$S_4$}
	\end{pspicture}
\end{center}
Moreover, by $K_n$ we will denote the complete graph on $n$ vertices.
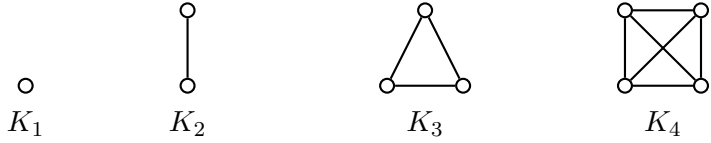
\begin{center}
	\begin{pspicture}(0,-1)(2,1.5)
		\cnode(0.5,0){3pt}{a1}
		\rput(0.5,-0.5){$K_1$}
	\end{pspicture}
	\begin{pspicture}(0,-1)(2,1.5)
		\cnode(0.5,0){3pt}{a1}
		\cnode(0.5,1){3pt}{a2}
		\psset{nodesep=0pt}
		\ncline{a1}{a2}
		\rput(0.5,-0.5){$K_2$}
	\end{pspicture}
	\begin{pspicture}(-1,-1)(2,1)
		\cnode(0,0){3pt}{a1}
		\cnode(1,0){3pt}{a2}
		\cnode(0.5,1){3pt}{a3}
		\psset{nodesep=0pt}
		\ncline{a1}{a2}
		\ncline{a1}{a3}
		\ncline{a2}{a3}
		\rput(0.5,-0.5){$K_3$}
	\end{pspicture}
	\begin{pspicture}(-1,-1)(2,1)
		\cnode(0,0){3pt}{a1}
		\cnode(1,0){3pt}{a2}
		\cnode(0,1){3pt}{a3}
		\cnode(1,1){3pt}{a4}
		\psset{nodesep=0pt}
		\ncline{a1}{a2}
		\ncline{a1}{a3}
		\ncline{a1}{a4}
		\ncline{a2}{a3}
		\ncline{a2}{a4}
		\ncline{a3}{a4}
		\rput(0.5,-0.5){$K_4$}
	\end{pspicture}
\end{center}
\begin{definition} 
	Let $\bA=(A,\varrho^\bA)$. We say that $\bA$ has property $(\star)$ if
	\[\text{For all } a,b,c\in A\text{ it holds that  } (a,b)\in\varrho^\bA, (b,c)\in\varrho^\bA \implies a=c.\tag{$\star$}\]
\end{definition}

The following Lemma is going to be the key in the classification of the connected po\-ly\-mor\-phism-ho\-mo\-ge\-ne\-ous graphs.
\begin{lemma}\label{star}
	Let $\bA=(A,\varrho^\bA)$ be a graph that does not have property $(\star)$. Then for every $k\ge 2$ there exists an $n\in\bN\setminus\{0\}$ such that $S_k$ is an induced subgraph of $\bA^n$.
\end{lemma}
\begin{proof}
	Since $\bA$ does not have property $(\star)$, there exist distinct  $a,b,c\in A$ such that $(a,b),(b,c)\in\varrho^\bA$. Let $k\ge 2$ and set $n:= k+1$. Define for $1\le i\le k$ tuples $\ta_i\in A^n$ according to:
	\begin{align*}
		\ta_1&=(a,c,a,a,\dots,a),\\
		\ta_2&=(a,a,c,a,\dots,a),\\
		&\quad\vdots\\
		\ta_k&=(a,a,a,\dots,a,c).
	\end{align*} 
	Clearly, the set $\{\ta_1,\dots,\ta_k\}$ is independent in $\bA^n$, since the first coordinate of each  $\ta_i$ is equal to $a$. Moreover, $\tb:=(b,b\dots,b)$ is connected to each $\ta_i$ in $\bA^n$, since $(a,b),(b,c)\in\varrho^\bA$. In other words, the set $\{\ta_1,\dots,\ta_k,\tb\}$ induces $S_k$ in $\bA^n$. 
\end{proof}

\begin{lemma}\label{center}
	Let $\bA=(A,\varrho^\bA)$ be a  po\-ly\-mor\-phism-ho\-mo\-ge\-ne\-ous graph that does not have property $(\star)$. Then every finite subset of vertices of $\bA$ has a common neighbor. 
\end{lemma}
\begin{proof}
	Suppose, there is a finite set $\{b_1,\dots,b_k\}$ of vertices of $\bA$ that has no common neighbor. By Lemma~\ref{star}, there exists some $n\in\bN\setminus\{0\}$ such that $\bA^n$ has an induced subgraph isomorphic to $S_k$.  Let $\{\ta_1,\dots,\ta_k,\tc\}$ be its vertex set such that $\tc$ is a common neighbor of the other vertices. For $1\le i\le k$, define $\tb_i\in A^n$ according to $\tb_i:=(b_i,\dots,b_i)$. Then the set $\{\tb_1,\dots,\tb_k\}$ does not have a common neighbor in $\bA^n$. However, then the mapping defined by $\ta_i\mapsto\tb_i$ ($1\le i\le k$) is a local homomorphism of $\bA^n$ that does not extend to an endomorphism. Hence $\bA^n$ is not ho\-mo\-mor\-phism-ho\-mo\-ge\-ne\-ous and thus $\bA$ is not po\-ly\-mor\-phism-ho\-mo\-ge\-ne\-ous --- contradiction. 
\end{proof}

\begin{proposition}[{Cameron, Ne\v{s}et\v{r}il \cite[Prop.2.1(a).]{CamNes06}}]\label{RadoSpan}
	A countably infinite graph $\bA$ contains the Rado graph as a spanning subgraph if and only if every finite set of vertices of $\bA$ has a common neighbor.
\end{proposition}

\begin{lemma}\label{smallcomponents}
	Let $\bA=(A,\varrho^\bA)$ be a connected graph with at least three vertices. Then $\bA$ does not have property $(\star)$. 
\end{lemma}
\begin{proof}
	If any two distinct vertices have distance $1$ then $\bA$ is a complete graph. In this case  for any three distinct vertices $a,b,c$ in $\bA$ we have $(a,b)\in\varrho^\bA$ and $(b,c)\in\varrho^\bA$. 
	
	So suppose that $\bA$ is not complete. Then there exist two vertices $a,c$ of distance $2$ in $\bA$. That is, there exists a $b\in A$ such that $(a,b)\in\varrho^\bA$ and $(b,c)\in\varrho^\bA$. 
\end{proof}

\begin{lemma}\label{trivialcase}
	Let $\bA=(A,\varrho)$ be a polymorphism-homogeneous graph that has property $(\star)$. Then either $\varrho^\bA=\emptyset$ or every connected component of $\bA$ is isomorphic to the complete graph $K_2$.
\end{lemma}
\begin{proof}
	By Lemma~\ref{smallcomponents} we have that the connected components of $\bA$ are either isomorphic to $K_1$ or to $K_2$.  
	
	Suppose $\bA$ has a connected component isomorphic to $K_1$ and one isomorphic to $K_2$. Let $\{a\}$ and $\{b,c\}$ be the vertex sets of connected components isomorphic to $K_1$ and $K_2$, respectively. Then the mapping $f:b\mapsto a$ is a local homomorphism of $\bA$. Any endomorphism of $\bA$ that extends $f$ would have to map $b$ to a neighbor of $a$. Since $a$ has no neighbors in $\bA$, $f$ can not be extended to an endomorphism. In particular, $\bA$ can not be polymorphism-homogeneous.
	
	It follows that in $\bA$ either all connected components are isomorphic to $K_1$ (in this case we have $\varrho^\bA=\emptyset$)  or each connected component of $\bA$ is isomorphism to $K_2$. 
\end{proof}

\begin{lemma}\label{k2k2}
	Let $\bA$ be a graph such that every connected component of $\bA$ is isomorphic to the complete graph $K_2$. Then for every $k\in\bN\setminus\{0\}$ every connected component of the graph $\bA^k$ is isomorphic to $K_2$, too. 
\end{lemma}
\begin{proof}
	Let $\ta,\tb,\tc$ be elements of $A^k$ such that $(\ta,\tb)\in\varrho^{\bA^k}$, and $(\tb,\tc)\in\varrho^{\bA^k}$. Suppose, $\ta=(a_1,\dots,a_k)$, $\tb=(b_1,\dots,b_k)$, and $\tc=(c_1,\dots,c_k)$. Since $(a_i,b_i)\in\varrho^\bA$, $(b_i,c_i)\in\varrho^\bA$, and since each connected component of $\bA$ is isomorphic to $K_2$, it follows that $a_i=c_i$, for all $i\in\{1,\dots,k\}$. In other words, we have $\ta=\tc$. Hence, $\bA^k$ has property $(\star)$ and  by Lemma~\ref{smallcomponents}, every connected component of $\bA^k$ is either isomorphic to $K_1$ or to $K_2$. 
	
	It remains to show that $\bA^k$ does not contain connected components isomorphic to $K_1$. 
	For this, let $\ta=(a_1,\dots,a_k)\in A^k$. Then for every $i\in\{1,\dots,k\}$ there is a unique vertex $b_i$ such that $\{a_i,b_i\}$ induce $K_2$ in $\bA$. Let $\tb:=(b_1,\dots,b_k)$. Then we have $(\ta,\tb)\in\varrho^{\bA^k}$. Thus $\ta$ lies in a connected component isomorphic to $K_2$ in $\bA^k$. 	
\end{proof}

\begin{theorem}
	A countable graph $\bA=(A,\varrho)$ is po\-ly\-mor\-phism-ho\-mo\-ge\-ne\-ous if and only if 
	\begin{enumerate}
		\item $\varrho^\bA=\emptyset$, or
		\item every connected component of $\bA$ is isomorphic to $K_2$, or
		\item $\bA$ has the Rado graph as a spanning subgraph.
	\end{enumerate}
\end{theorem}
\begin{proof}
	 ``$\Rightarrow$'' 	If $\bA$ does not have property $(\star)$, then we have by Lemma~\ref{center} and by Proposition~\ref{RadoSpan} that $\bA$ contains the Rado graph as a spanning subgraph.
	
	If $\bA$ has property $(\star)$, then we conclude from Lemma~\ref{trivialcase} that either $\varrho^\bA$ is empty or that every connected component of $\bA$ is isomorphic to $K_2$. 
	
	``$\Leftarrow$'' Suppose, that $\varrho^\bA$ is empty. Then we also have $\varrho^{\bA^k}=\emptyset$, for every $k\in\bN\setminus\{0\}$. In other words, we have for every $k\in\bN\setminus\{0\}$ that every connected component of $\bA^k$ is isomorphic to the complete graph $K_1$. Thus $\bA^k$ is homomorphism homogeneous, and consequently, by Proposition~\ref{PHHH}, $\bA$ is polymorphism-homogeneous.
	
	Suppose that every connected component of $\bA$ is isomorphic to the complete graph $K_2$. Then, by Lemma~\ref{k2k2}, we have for every $k\in\bN\setminus\{0\}$ that all the connected components of the graph $\bA^k$ are isomorphic to $K_2$. Thus, $\bA^k$ is homomorphism homogeneous for every $k\in\bN\setminus\{0\}$. By Proposition~\ref{PHHH}, we have that $\bA$ is polymorphism-homogeneous.    
		
	Suppose now that $\bA$ has the Rado graph as a spanning subgraph. By Proposition~\ref{RadoSpan},  every finite set of vertices of $\bA$ has a common neighbor. Let $k\in\bN\setminus\{0\}$, and let $\{\ta_1,\dots,\ta_n\}$ be a set of vertices of $\bA^k$. For $1\le i\le n$, suppose, $\ta_i=(a_{i,1},\dots,a_{i,k})$. For $j\in\{1,\dots,k\}$, let $c_j$ be a common neighbor of $\{a_{1,j},\dots, a_{n,j}\}$, and set $\tc:=(c_1,\dots,c_k)$. Then $\tc$ is a common neighbor of $\{\ta_1,\dots,\ta_n\}$. Hence, by Proposition~\ref{RadoSpan}, $\bA^k$  has the Rado graph as a spanning subgraph. In particular, $\bA^k$ is ho\-mo\-mor\-phism-ho\-mo\-ge\-ne\-ous. From Proposition~\ref{PHHH} it follows that $\bA$ is po\-ly\-mor\-phism-ho\-mo\-ge\-ne\-ous.  
\end{proof}

\subsection*{Polymorphism-homogeneous posets}
A poset is a relational structure $\bA=(A,\le)$ where $\le$ is a binary reflexive, antisymmetric transitive relation. $\bA$ is called \emph{trivial} if $a\le b$ implies $a=b$ --- in other words, it is an anti-chain. 
For $X,Y\subseteq A$ and $x,y\in A$ we write
\begin{align*}
	x\le Y \quad& \text{if}\quad\forall y\in Y\, (x\le y),\\
	X\le y \quad&\text{if}\quad\forall x\in X\, (x\le y),\\
	X\le Y \quad&\text{if}\quad\forall x\in X\,\forall y\in Y\, (x\le y).
\end{align*}

We distinguish the following posets:
\begin{center}
	\begin{pspicture}(-1,-1)(2,1.5)
		\cnode(0,0){3pt}{a1}
		\cnode(1,0){3pt}{a2}
		\cnode(0,1){3pt}{a3}
		\cnode(1,1){3pt}{a4}
		\psset{nodesep=0pt}
		\ncline{a1}{a3}
		\ncline{a1}{a4}
		\ncline{a2}{a3}
		\ncline{a2}{a4}
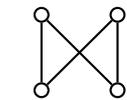
		\rput(0.5,-0.5){the bow-tie}
	\end{pspicture}
	\begin{pspicture}(-1,-1)(2,1)
		\cnode(0,0){3pt}{a1}
		\cnode(1,0){3pt}{a2}
		\cnode(0,1){3pt}{a3}
		\cnode(1,1){3pt}{a4}
		\cnode(0.5,0.5){3pt}{c}
		\psset{nodesep=0pt}
		\ncline{a1}{c}
		\ncline{c}{a4}
		\ncline{c}{a3}
		\ncline{a2}{c}
		\rput(0.5,-0.5){$X_5$}
	\end{pspicture}
\end{center}

\begin{definition}
	Let $\bA=(A,\le)$ be a poset. Then $\bA$ is called \emph{locally bounded} if for every finite subset $B$ of $A$ there exist $c,d\in A$ such that $c\le B\le d$. 
\end{definition}

\begin{definition}
	Let $\bA=(A,\le)$ be a poset. Then $\bA$ is called \emph{$X_5$-dense} if for all  $a_1,a_2,a_3,a_4\in A_4$  with $\{a_1,a_2\}\le\{a_3,a_4\}$, there exists a $c\in A$ such that $\{a_1,a_2\}\le c\le \{a_3,a_4\}$.
\end{definition}

The ho\-mo\-mor\-phism-ho\-mo\-ge\-ne\-ous posets were completely characterized by Ma\v sulovi\'c \cite{Mas07} and, independently, by Cameron and Lockett \cite{CamLoc10}:
\begin{theorem}[{Ma\v{s}ulovi\'c \cite[Thm.4.5]{Mas07}}]\label{posetHH}
	A poset $\bA$ is ho\-mo\-mor\-phism-ho\-mo\-ge\-ne\-ous if and only if
	\begin{enumerate}
		\item every connected component of $\bA$ is a chain, or
		\item $\bA$ is a tree, or
		\item $\bA$ is a dual tree, or
		\item $\bA$ is locally bounded and $X_5\notin\Age(\bA)$, or
		\item $\bA$ is locally bounded and $X_5$-dense.
	\end{enumerate}
\end{theorem}

\begin{lemma}\label{X5dense}
	Let $\bA=(A,\le)$ be a poset and let $k\in\bN\setminus\{0\}$. Then $\bA$ is $X_5$-dense if and only if $\bA^k$ is $X_5$-dense.
\end{lemma}
\begin{proof}
	``$\Rightarrow$''
	Let $k\ge 1$, and let $\ta_1,\ta_2,\ta_3,\ta_4\in A^k$ such that $\{\ta_1,\ta_2\}\le\{\ta_3,\ta_4\}$ in  $\bA^k$. Let us say $\ta_i=(a_{i,1},\dots,a_{i,k})$ for $i\in\{1,2,3,4\}$. Then for all $j\in\{1,\dots,k\}$ we have $\{a_{1,j},a_{2,j}\}\le\{a_{3,j},a_{4,j}\}$ in $\bA$. By the assumption, for every $j\in\{1,\dots,k\}$ there exists a $c_j$ such that $\{a_{1,j},a_{2,j}\}\le c_j\le\{a_{3,j},a_{4,j}\}$. But then, with $\tc=(c_1,\dots,c_k)$, we also have 	$\{\ta_1,\ta_2\}\le\tc\le\{\ta_3,\ta_4\}$. Hence $\bA^k$ is $X_5$-dense.
	
	``$\Leftarrow$'' Let $a_1,a_2,a_3,a_4\in A$ such that $\{a_1,a_2\}\le\{a_3,a_4\}$ in $\bA$.
	Define $\ta_i\in A^k$ according to $\ta_i=(a_i,\dots,a_i)$, where $i\in\{1,2,3,4\}$. Clearly, we have $\{\ta_1,\ta_2\}\le\{\ta_3,\ta_4\}$ in $\bA^k$. Thus, by the assumption, there exists $\tc\in A^k$ such that $\{\ta_1,\ta_2\}\le\tc\le\{\ta_3,\ta_4\}$ in $\bA^k$. Suppose $\tc=(c_1,\dots,c_k)$. Then $\{a_1,a_2\}\le c_1\le\{a_3,a_4\}$  in $\bA$. Hence $\bA$ is $X_5$-dense.  
\end{proof}

\begin{lemma}\label{locbound}
	Let $\bA=(A,\le)$ be a poset and let $k\in\bN\setminus\{0\}$. Then $\bA$ is locally bounded if and only if $\bA^k$ is locally bounded.\qed
\end{lemma}
\begin{proof}
	``$\Rightarrow$'' Let $\{\ta_1,\dots,\ta_m\}\subseteq A^k$, where $\ta_i=(a_{i,1},\dots,a_{i,k})$ for $i\in\{1,\dots,m\}$. By the assumption, we have that for every $j\in\{1,\dots,k\}$ there exist $c_j,d_j\in A$ such that $c_j\le\{a_{1,j},\dots,a_{m,j}\}\le d_j$. Define $\tc:=(c_1,\dots,c_k)$ and $\td:=(d_1,\dots,d_k)$. Then $\tc\le\{\ta_1,\dots,\ta_m\}\le\td$. Hence, $\bA^k$ is locally bounded.
	
	``$\Leftarrow$'' Let $\{a_1,\dots,a_m\}\subseteq A$. For $i\in\{1,\dots,m\}$ define $\ta_i\in A^k$ according to $\ta_i:=(a_i,\dots,a_i)$. By the assumption, there exist a $\tc,\td\in A^k$ such that $\tc\le\{\ta_1,\dots,\ta_m\}\le\td$. Suppose, $\tc=(c_1,\dots,c_k)$, $\td=(d_1,\dots,d_k)$. Then we have $c_1\le\{a_1,\dots,a_m\}\le d_1$. Hence, $\bA$ is locally bounded.
\end{proof}

\begin{lemma}\label{chain}
	Let $\bA=(A,\le)$ be a poset that contains a chain of length $l$. Then for every $l$-element poset $\bB$ there exists a $k\in\bN\setminus\{0\}$ such that $\bB\in\Age(\bA^k)$.
\end{lemma}
\begin{proof}
	By the Dushnik-Miller Theorem \cite{DusMil41}, the order relation of $\bB$ is equal to the intersection of a collection  of linear order relations on $B$. The smallest possible size of such a collection is called the order dimension of $\bB$. In other words, if the order dimension of $\bB$ is equal to $k$, then there exists an $l\times k$-matrix
	\[
	\begin{pmatrix}
		a_{1,1} & a_{1,2}&\dots & a_{1,k}\\
		a_{2,1} & a_{2,2}&\dots & a_{2,k}\\
		\vdots & \vdots &\ddots & \vdots \\
		a_{l,1} & a_{l,2}&\dots & a_{l,k}
	\end{pmatrix}
	\]
	with entries from $B$ such that 
	\begin{enumerate}
		\item $\forall i\in\{1,\dots,k\}: \{a_{1,i},\dots,a_{l,i}\}=B$,
		\item $\forall a,b\in B:  a\le b \iff \Big(\forall j\in\{1,\dots,k\}\,:\,h_j(a)\le h_j(b)\Big) $, 
	\end{enumerate}
	where $h_j(x)=i$ whenever $a_{i,j}=x$.
	
	Let $c_1< c_2<\dots<c_l$ be a chain of length $l$ in $\bA$. Define $f:B\to A^m$ by $a\mapsto (c_{h_1(a)},\dots,c_{h_k(a)})$. Then, by the second property of the matrix $M$, we have that $f$ is actually an embedding of $\bB$ into $\bA^k$. In particular, $\bB\in\Age(\bA^k)$.  
\end{proof}

\begin{lemma}\label{largeage}
	Let $\bA=(A,\le)$ be a  poset that is not an antichain. Then for every finite poset $\bB$ there exists a $k\in\bN\setminus\{0\}$ such that $\bB\in\Age(\bA^k)$.\qed 
\end{lemma}
\begin{proof}
	Since $\bA$ is not an antichain, it contains elements $a$ and $b$ such that $a<b$. Suppose that $\bB$ has $l$ elements (we may assume that $l>1$ since otherwise nothing needs to be proved). Then $\bA^{l-1}$ contains a chain of length $l$ --- namely
	\[ (a,a,\dots,a)<(a,a,\dots,a,b)<(a,a,\dots,b,b)<\dots<(b,b\dots,b).\]
	Hence, by Lemma~\ref{chain}, there exists some $m$ such that $\bB\in \Age((\bA^{l-1})^m)$. Since $(\bA^{l-1})^m\cong \bA^{(l-1)\cdot m}$, with $k:=(l-1)m$, we have $\bB\in\Age(\bA^k)$.  
\end{proof}

\begin{lemma}\label{locboundPH}
	Let $\bA=(A,\le)$ be a  po\-ly\-mor\-phism-ho\-mo\-ge\-ne\-ous poset that is not an antichain. Then $\bA$ is locally bounded.
\end{lemma}
\begin{proof}
	Suppose on the contrary that $\bA$ is not locally bounded. Without loss of generality assume that $\bA$ is not upwards directed. Let $a,b$ be two  elements of $\bA$ without a joint upper bound. Since $\bA$ is not an antichain, by Lemma~\ref{largeage}, there exists a $k$, such that in $\bA^k$ there exist two non-comparable elements $\tc,\td$ that have a joint upper bound $\tu$. Let $\ta,\tb\in A^k$ defined according to $\ta:=(a,\dots,a)$ and $\tb:=(b,\dots,b)$. Then $\ta$ and $\tb$ have no joint upper bound in $\bA^k$. Note that the mapping $f: \tc\mapsto\ta,\,\td\mapsto\tb$ is a local homomorphism of $\bA^k$ which can not be extended to an endomorphism of $\bA^k$. It follows that $\bA^k$ is not ho\-mo\-mor\-phism-ho\-mo\-ge\-ne\-ous. Hence, by Proposition~\ref{PHHH}, $\bA$ is not po\-ly\-mor\-phism-ho\-mo\-ge\-ne\-ous --- contradiction. 
\end{proof}

\begin{lemma}\label{X5densePH}
	Let $\bA=(A,\le)$ be a po\-ly\-mor\-phism-ho\-mo\-ge\-ne\-ous poset. Then $\bA$ is $X_5$-dense. 
\end{lemma}
\begin{proof}
	Suppose on the contrary that $\bA$ is not $X_5$-dense. Then there exist four elements $a_1,a_2,a_3,a_4\in A$ such that $\{a_1,a_2\}\le\{a_3,a_4\}$, but there is no $c\in A$ such that $\{a_1,a_2\}\le c\le\{a_3,a_4\}$. In this case, $\{a_1,a_2,a_3,a_4\}$ induce a substructure of $\bA$ that is isomorphic to the bow-tie. In particular, $\bA$ is not an antichain. From Lemma~\ref{largeage} it follows that there is a $k\in\bN\setminus\{0\}$ such that there are $\tb_1,\tb_2,\tb_3,\tb_4,\td\in A^k$ which fulfill $\{\tb_1,\tb_2\}\le\td\le\{\tb_3,\tb_4\}$.  Now, for $i\in\{1,2,3,4\}$ define $\ta_i\in A^k$ according to $\ta_i=(a_i,\dots,a_i)$. Then there is no $\tc\in A^k$ such that $\{\ta_1,\ta_2\}\le\tc\le\{\ta_3,\ta_4\}$.
	Note that the mapping $f: \tb_i\mapsto \ta_i$ ($i\in\{1,2,3,4\}$) is a local homomorphism of $\bA^k$ which can not be extended to an endomorphism of $\bA^k$. We conclude that $\bA^k$ is not ho\-mo\-mor\-phism-ho\-mo\-ge\-ne\-ous and thus, by Proposition~\ref{PHHH}, $\bA$ is not po\-ly\-mor\-phism-ho\-mo\-ge\-ne\-ous --- contradiction.
\end{proof}
Now we are ready to give a complete characterization of the po\-ly\-mor\-phism-ho\-mo\-ge\-ne\-ous posets. Note that the characterization is independent of the cardinality of the posets in question.
\begin{theorem}
	A partially ordered set $\bA=(A\le)$ is po\-ly\-mor\-phism-ho\-mo\-ge\-ne\-ous if and only if either it  is an antichain, or it is locally bounded and $X_5$-dense.
\end{theorem}
\begin{proof}
	``$\Rightarrow$'' If $\bA$ is a polymorphism-homogeneous structure, then we know, by Lemma~\ref{X5densePH}, that $\bA$ is $X_5$-dense. Moreover, by Lemma~\ref{locboundPH}, $\bA$ is either an antichain or it is locally bounded.
	
	``$\Leftarrow$'' If $\bA$ is an antichain, then every function on $A$ is a polymorphism of $\bA$. Thus, $\bA$ is polymorphism-homogeneous. Suppose now that $\bA$ is locally bounded and $X_5$-dense. By Lemma~\ref{locbound} and Lemma~\ref{X5dense}, $\bA^k$ is locally bounded and $X_5$-dense, for every $k\in\bN\setminus\{0\}$. Finally, by Theorem~\ref{posetHH}, $\bA^k$ is homomorphism homogeneous, for every $k\in\bN\setminus\{0\}$. Thus, by Proposition~\ref{PHHH}, $\bA$ is polymorphism-homogeneous.
\end{proof}
As a special case, we obtain the following characterization for finite po\-ly\-mor\-phism-ho\-mo\-ge\-ne\-ous posets:
\begin{corollary}
	A finite poset $\bA$ is po\-ly\-mor\-phism-ho\-mo\-ge\-ne\-ous if and only if it is an antichain or a lattice-order.
\end{corollary}
\begin{proof}
	This follows directly from \cite[Cor.4.5]{Mas07}, where the finite homomorphism homogeneous posets are characterized. In particular, it is shown there that a finite locally bounded and $X_5$-dense poset is actually a lattice order.  On the other hand, every lattice order is clearly locally bounded and $X_5$-dense.
	
\end{proof}

\subsection*{Polymorphism-homogeneous strict posets}
A strict poset is a relational structure $\bA=(A,<)$, where $<$ is a binary asymmetric, and transitive relation. For a set $M\subseteq A$ we define 
\begin{align*}
	\bA_{<M} &:= \{a\in A\mid a < M\},\\
	\bA_{>M} &:= \{a\in A\mid a > M\}.
\end{align*}
\begin{definition}
	Let $\bA=(A,<)$ be a strict poset. We say that  $\bA$ is \emph{$X_5$-dense} if for every $\{a_1,a_2,a_3,a_4\}\in A$ with $\{a_1,a_2\}<\{a_3,a_4\}$ there exists a $c\in A$ such that $\{a_1,a_2\}<c<\{a_3,a_4\}$. Moreover, we say that $\bA$ is \emph{strictly locally bounded} if for every finite set $X\subseteq A$ we have that neither $\bA_{<X}\neq\emptyset$ nor $\bA_{>X}\neq\emptyset$ is empty.
\end{definition}

The countable ho\-mo\-mor\-phism-ho\-mo\-ge\-ne\-ous strict posets were completely characterized by Cameron and Lockett:
\begin{theorem}[{Cameron, Lockett \cite[Prop.15]{CamLoc10}}]\label{strictHHposets}
	A countable strict po\-set $\bA=(A,<)$ is ho\-mo\-mor\-phism-ho\-mo\-ge\-ne\-ous if and only if
	\begin{enumerate}
		\item it is an antichain, or
		\item \label{class2}it is a disjoint union of copies of $(\mathbb{Q},<)$, or
		\item \label{class3} it is a tree with no minimal elements such that for all finite $Q\subseteq A$, $\bA_{<Q}$ has no maximal elements, or
		\item \label{class4} it is a dual tree with no maximal elements such that for all finite $Q\subseteq A$, $\bA_{>Q}$ has no minimal elements, or
		\item\label{U} $\bA$ is strictly locally bounded, for all finite $X\subseteq A$, $\bA_{<X}$ has no maximal elements and $\bA_{>X}$ has no minimal elements, and it is $X_5$-dense, or
		\item \label{class6}it is strictly locally bounded,  for all finite $X\subseteq A$, $\bA_{<X}$ has no maximal elements and $\bA_{>X}$ has no minimal elements, and $X_5\notin\Age(\bA)$.
	\end{enumerate}
\end{theorem}
\begin{proposition}[{Cameron, Lockett \cite[Prop.13]{CamLoc10}}]\label{genposetchar}
	The posets from Theorem~\ref{strictHHposets}(\ref{U}) are exactly the extensions of the countable universal homogeneous poset $\bU=(U,<)$.  
\end{proposition}

\begin{lemma}\label{infChain}
	Every countable homomorphism homogeneous strict poset that is not an antichain contains an  infinite chain. 
\end{lemma}
\begin{proof}
	Let $\bA=(A,<)$ be a homomorphism homogeneous strict poset that is not an antichain. We will go systematically through the classes of homomorphism homogeneous strict posets given in Theorem~\ref{strictHHposets}.
	
	If $\bA$ belongs to class~\eqref{class2}, then any copy of $\mathbb{Q}$ is an infinite chain in $\bA$. 
	
	Suppose, $\bA$ belongs to class~\eqref{class3}. Since $\bA$ has no minimal elements, it contains an infinite descending chain. Analogously, if $\bA$ belongs to class~\eqref{class4}, then it contains an infinite ascending chain.
	
	Suppose finally, that $\bA$ belongs to class~\eqref{U} or class~\eqref{class6}. Let $X\subseteq A$ be finite. Then $\bA_{<X}\neq\emptyset$, since $\bA$ is strictly locally bounded. Moreover, $\bA_{<X}$ has no maximal elements. Hence, $\bA_{<X}$ contains an infinite ascending chain. 
\end{proof}

\begin{lemma}\label{x5dense2}
	Let $\bA=(A,<)$ be a strict poset and let $k\in\bN\setminus\{0\}$. Then $\bA$ is $X_5$-dense if and only if $\bA^k$ is $X_5$-dense.
\end{lemma}
\begin{proof}
	Cf. the proof of Lemma~\ref{X5dense}.
\end{proof}

\begin{lemma}\label{locbound2}
	Let $\bA=(A,<)$ be a strict poset and let $k\in\bN\setminus\{0\}$. Then $\bA$ is strictly locally bounded if and only if $\bA^k$ is strictly locally bounded.
\end{lemma}
\begin{proof}
	Cf. the proof of Lemma~\ref{locbound}.
\end{proof}

\begin{lemma}\label{chainstrict}
	Let $\bA=(A,<)$ be a strict poset that contains a chain of length $l$. Then for every $l$-element poset $\bB$ there exists some $k\in\bN\setminus\{0\}$ such that $\bB\in\Age(\bA^k)$.
\end{lemma}
\begin{proof}
	Cf. the proof of Lemma~\ref{chain}.
\end{proof}

\begin{lemma}\label{partchar}
	Let $\bA$ be a  po\-ly\-mor\-phism-ho\-mo\-ge\-ne\-ous strict poset that is not an antichain. Then $\bA$ is strictly locally bounded and $X_5$-dense.
\end{lemma} 
\begin{proof}
	Essentially the same proofs as in Lemma~\ref{locboundPH} and Lemma~\ref{X5densePH} work. Instead of using Lemma~\ref{largeage}, we use Lemma~\ref{infChain}, in conjunction with Lemma~\ref{chainstrict}. 
\end{proof}

Now we are ready to give a complete characterization of the countable po\-ly\-mor\-phism-ho\-mo\-ge\-ne\-ous strict posets:
\begin{theorem}
	A countable strict partially ordered set $\bA=(A,<)$ is po\-ly\-mor\-phism-ho\-mo\-ge\-ne\-ous if and only if either it  is an antichain, or it is an extension of the countable universal homogeneous poset $\bU=(U,<)$.
\end{theorem}
\begin{proof}
	``$\Rightarrow$'' Suppose that $\bA$ is polymorphism-homogeneous. Then, by Lemma~\ref{partchar}, we have that $\bA$ is either an antichain or it is strictly locally bounded and $X_5$-dense. Suppose that $\bA$ is not an antichain. Since $\bA$ is polymorphism-homogeneous, it is in particular homomorphism homogeneous. Thus, by Theorem~\ref{strictHHposets} and by Proposition~\ref{genposetchar}, $\bA$ is an extension of the countable universal homogeneous poset $\bU$.
	
	``$\Leftarrow$'' Clearly, if $\bA$ is an antichain, then it is polymorphism-homogeneous. So suppose that $\bA$ is an extension of $\bU$. We will show that for every $k\in\bN\setminus\{0\}$, $\bA^k$ is an extension of $\bU$, too.  
	
	Since $\bA$ is strictly locally bounded and $X_5$-dense, it follows from Lemma~\ref{x5dense2} and from Lemma~\ref{locbound2}, that $\bA^k$ is strictly locally bounded and $X_5$-dense, for all $k\in\bN\setminus\{0\}$. It remains to show, that for every $k\in\bN\setminus\{0\}$, and for all finite $X\subseteq A^k$ we have that $\bA^k_{<X}$ has no maximal elements and $\bA^k_{>X}$ has no minimal elements. 
	
	So let $k\in\bN\setminus\{0\}$ and let $X\subseteq A^k$ be finite. Assume $X=(\tx_1,\dots,\tx_n)$ and for $1\le i\le n$ suppose $\tx=(x_{i,1},\dots,x_{i,k})$. Now for $j\in\{1,\dots,k\}$ define $X_j:=\{x_{i,j}\mid 1\le i \le n\}$. Let $(y_1,\dots, y_k)\in \bA^k_{<X}$ (this exists, since $\bA^k$ is strictly locally bounded). Then we have $y_j\in  \bA_{<X_j}$ for $j\in\{1,\dots,k\}$. Since, by assumption,  $\bA_{<X_j}$ has no maximal elements, there exists $y'_j\in \bA_{<X_j}$ such that $y_j< y'_j$, for all $j\in\{1,\dots,k\}$. Thus $(y_1,\dots,y_k)<(y'_1,\dots,y'_k)$ and $(y'_1,\dots,y'_k)\in\bA^k_{<X}$. Consequently, $\bA^k_{<X}$ has no maximal elements. Analogously it is shown that $\bA^k_{>X}$ has no minimal elements. Now, from Proposition~\ref{genposetchar}, it follows that $\bA^k$ is an extension of $\bU$. 

	By Theorem~\ref{strictHHposets}, together with Proposition~\ref{genposetchar}, it follows that $\bA^k$ is homomorphism homogeneous, for all $k\in\bN\setminus\{0\}$. Thus, from Proposition~\ref{PHHH}, it follows that $\bA$ is polymorphism-homogeneous.	
\end{proof}

\subsection*{Polymorphism-ho\-mo\-ge\-ne\-ous lattices of equivalence relations}
Let $A$ be a set and let $\cE(A)$ be the set of all equivalence relations on $A$. Then $\cE(A)$, ordered by inclusion forms a complete lattice, in which the infimum of a set of equivalence relations is their intersection, and the supremum is the transitive closure of their union.

A sublattice $\cL$ of $\cE(A)$ is called \emph{meet-complete} if it is closed with respect to arbitrary intersections. It is called \emph{arithmetical} if 
\begin{enumerate}
	\item $\forall \theta_1,\theta_2\in\cL\,:\, \theta_1\circ\theta_2=\theta_2\circ\theta_1$,
	\item $\forall \theta_1,\theta_2,\theta_2\in\cL\,:\, \theta_1\land(\theta_2\lor\theta_3)=(\theta_1\land\theta_2)\lor(\theta_1\land\theta_3)$
\end{enumerate}

We will call a sublattice $\cL$ of $\cE(A)$ po\-ly\-mor\-phism-ho\-mo\-ge\-ne\-ous if its canonical structure $\bC_{\cL}$ is po\-ly\-mor\-phism-ho\-mo\-ge\-ne\-ous. The following beautiful characterization of po\-ly\-mor\-phism-ho\-mo\-ge\-ne\-ous meet-complete sublattices of $\cE(A)$ is due to Kaarli \cite{Kaa83}:
\begin{theorem}[{Kaarli \cite[Thm.3]{Kaa83}}]
	Let $A$ be a countable set and let $\cL$ be a meet-complete sublattice of $\cE(A)$. Then $\cL$ is po\-ly\-mor\-phism-ho\-mo\-ge\-ne\-ous if and only if it is arithmetical.
\end{theorem}

%\bibliographystyle{abbrv} 
%\bibliography{PP11}

\end{document}